\documentclass[11pt]{article}

\usepackage[margin=1.5in]{geometry} 
\usepackage{amsmath}
\usepackage{amsthm}
\usepackage{amssymb}
\usepackage{graphicx}
\usepackage{hyperref} 
\usepackage{mathtools} 
\usepackage[hypcap]{caption}

\newtheorem{thm}{Theorem}[section]
\newtheorem{cor}[thm]{Corollary}
\newtheorem{lem}[thm]{Lemma}
\newtheorem{prop}[thm]{Proposition}
 
\newtheorem{obs}[thm]{Observation}

\newtheorem*{thm*}{Theorem}
\newtheorem*{conj*}{Conjecture}
\newtheorem*{lem*}{Lemma}
\newtheorem*{obs*}{Observation}

\theoremstyle{definition}

\theoremstyle{remark}

\newtheorem*{claim}{Claim}

\newcommand{\Nn}{\mathcal N}

\newcommand{\Cc}{\mathcal C}

\newcommand{\Ee}{\mathcal E}

\newcommand{\Dd}{\mathcal D}
\newcommand{\Bb}{\mathcal B}

\newcommand{\Mm}{\mathcal M}
\newcommand{\Uu}{\mathcal U}

\newcommand{\eq}{{=}}
\newcommand{\iso}{\cong}

\newcommand{\ie}{\textit{i.e.}}
\newcommand{\Ie}{\textit{I.e.}}
\newcommand{\iiff}{if and only if }
\newcommand{\wolog}{without loss of generality}

\DeclareMathOperator{\cl}{cl}  
\DeclareMathOperator{\co}{co}

\newcommand{\onesum}{\oplus} 
\newcommand{\twosum}{\oplus_2}

\newcommand{\twosume}[2]{\prescript{#1}{}{\oplus}_2^{#2} \:}

\newcommand{\gltwosum}[1]{\overset{#1}{\oplus_2} U_{2,4}}

\newcommand{\br}[1]{\left( #1 \right)} 
\newcommand{\dcup}{\mathbin{\dot\cup}}
\newcommand{\Om}{\Omega}
\newcommand{\GB}{\ensuremath{(G,\Bb)}}

\newcommand{\GBp}{(G',\Bb')}

\graphicspath{{../Figures/}}

\title{On excluded minors of connectivity 2 for the class of frame matroids}
\author{
   Matt DeVos\thanks{Department of Mathematics, Simon Fraser University, 8888 University Drive, Burnaby, BC, Canada. Email: mdevos@sfu.ca.
     Supported in part by an NSERC Discovery Grant (Canada).}
\and
   Daryl Funk\thanks{School of Mathematics and Statistics, Victoria University of Wellington, Wellington, New Zealand. Email: daryl.funk@vuw.ac.nz.}
\and
   Irene Pivotto\thanks{School of Mathematics and Statistics, University of Western Australia, 35 Stirling Highway, Crawley, WA, Australia. Email: irene.pivotto@uwa.edu.au.}
}

\date{} 

\begin{document}

\maketitle

\begin{abstract} 
We investigate the set of excluded minors of connectivity 2 for the class of frame matroids.  
We exhibit a list $\Ee$ of 18 such matroids, and show that if $N$ is such an excluded minor, then either $N \in \Ee$ or $N$ is a 2-sum of $U_{2,4}$ and a 3-connected non-binary frame matroid.  
\end{abstract}



\section{Introduction}

A matroid is \emph{frame} if it may be extended so that it contains a basis $B$ (its \emph{frame}) such that every element is spanned by two elements of $B$.  
Frame matroids are a natural generalisation of graphic matroids.  
Indeed, the cycle matroid $M(G)$ of a graph $G \eq (V,E)$ is naturally extended by adding $V$ as a basis, and declaring each non-loop edge to be minimally spanned by its endpoints.  
Zaslavski \cite{MR1273951} has shown that the class of frame matroids is precisely that of matroids arising from \emph{biased graphs} (whence these have also been called \emph{bias} matroids), as follows.  
A \emph{biased graph} $\Omega$ consists of a pair $(G, \mathcal{B})$, where $G$ is a graph and $\mathcal{B}$ is a collection of cycles of $G$, called \emph{balanced}, such that no theta subgraph contains exactly two balanced cycles; a \emph{theta} graph consists of a pair of distinct vertices and three internally disjoint paths between them.  
We say such a collection $\Bb$ satisfies the \emph{theta property}. 
The membership or non-membership of a cycle in $\Bb$ is its \emph{bias}; cycles not in $\mathcal{B}$ are \emph{unbalanced}.  

\label{pagerefforbiaedgraphconstruction}
Let $M$ be a frame matroid on ground set $E$, with frame $B$.  
By adding elements in parallel if necessary, we may assume $B \cap E = \emptyset$.  
Hence for some matroid $N$, $M = N \setminus B$ where $B$ is a basis for $N$ and every element $e \in E$ is spanned by a pair of elements in $B$.  
Let $G$ be the graph with vertex set $B$ and edge set $E$, in which $e$ is a loop with endpoint $f$ if $e$ is in parallel with $f \in B$, and otherwise $e$ is an edge with endpoints $f, f' \in B$ if $e \in \cl\{f,f'\}$.  
Setting $\Bb = \{ C \mid C$ is a cycle for which $E(C)$ is a circuit of $M \}$ yields a biased graph $\GB$, and the circuits of $M$ are precisely those sets of edges inducing one of: 
(1)  a balanced cycle, 
(2)  two edge-disjoint unbalanced cycles intersecting in just one vertex, 
(3)  two vertex-disjoint unbalanced cycles along with a path connecting them, or 
(4)  a theta subgraph in which all three cycles are unbalanced \cite{MR1273951}.  
We call a subgraph as in (2) or (3) a \emph{pair of handcuffs}, \emph{tight} or \emph{loose}, respectively.  
We say such a biased graph $\GB$ \emph{represents} the frame matroid $M$, and write $M = F\GB$.  

Observe that for a biased graph $\GB$, if $\mathcal{B}$ contains all cycles in $G$, then $F(G, \mathcal{B})$ is the cycle matroid $M(G)$ of $G$.  
We therefore view a graph as a biased graph with all cycles balanced.  
At the other extreme, when no cycles are balanced $F(G, \emptyset)$ is the bicircular matroid of $G$, introduced by Sim\~{o}es-Pereira \cite{MR0317973} and further investigated by Matthews \cite{MR0505702}, Wagner \cite{MR815399}, and others 
(for instance, \cite{MR3100270, MR1892972}).  
Frame matroids also include Dowling geometries \cite{MR0307951} (see also \cite{MR2017726}).  

A class of matroids is \emph{minor-closed} if every minor of a matroid in the class is also in the class.  
For any minor-closed family, there is a set of \emph{excluded minors} consisting of those matroids not in the family all of whose proper minors are in the family.  
The class of frame matroids is minor-closed.  
Little is known about excluded minors for the class of frame matroids; 
Zaslavski has exhibited several in \cite{MR1273951}.  
The class of bicircular matroids is minor-closed; DeVos, Goddyn, Mayhew, and Royle \cite{Matt_Luis_exminorsforbicircular} have shown that an excluded minor for the class of bicircular matroids has less than 16 elements, 
and thus that the set of excluded minors for this class is finite.  
Perhaps, like graphic and bicircular matroids, the larger class of frame matroids may also be characterised by a finite list of excluded minors.  
On the other hand, as we have shown elsewhere \cite{MR3267062}, there are natural minor-closed families of frame matroids whose sets of excluded minors are infinite.  
Perhaps, like the class of matroids representable over the reals, the set of excluded minors for frame matroids is infinite.  
In this paper, we begin by seeking to determine those excluded minors for the class of frame matroids that are not 3-connected.  
We come close, determining a set $\Ee$ of 18 particular excluded minors for the class, and show that any other excluded minor of connectivity 2 for the class has a special form.  
We prove: 

\begin{thm} \label{thm:two_sep_ex_min_main} 
Let $M$ be an excluded minor for the class of frame matroids, and suppose $M$ is not 3-connected.  
Then either $M$ is isomorphic to a matroid in $\Ee$ or $M$ is the 2-sum of a 3-connected non-binary frame matroid and $U_{2,4}$.  
\end{thm}

The remainder of this paper is organised as follows.  
We first discuss some of the key concepts we need for our investigation.  
In Section \ref{sec:2-sums of frame matroids} we discuss 2-sums of frame matroids and of biased graphs, and provide a characterisation of when a 2-sum of two frame matroids is frame.  
This is enough for us to determine the first nine excluded minors on our list, and to drastically narrow our search for more.  
These tasks are accomplished in Section \ref{sec:Excluded_minors}.  
In particular, we investigate some key properties any excluded minor not yet on our list must have.  
In Section \ref{sec:Main_theorem} we complete the proof of Theorem \ref{thm:two_sep_ex_min_main}, determining the remaining excluded minors in our list.   

Theorem \ref{thm:two_sep_ex_min_main} give a strong structural description of excluded minors that are not 3-connected.  
However, the investigation remains incomplete --- the final case remaining is to determine those excluded minors of the form captured in the second part of the statement of Theorem \ref{thm:two_sep_ex_min_main}.  
We anticipate that the analysis required to complete this final case will be at least as long and technical as that required here, and that the result will be at least a doubling of the number of excluded minors on our list, but that the list will remain finite.  

We close this preliminary section by noting that in the course of proving Theorem \ref{thm:two_sep_ex_min_main}, we discover an operation analogous to a Whitney twist in a graph, which we call a \emph{twisted flip}.  
Just as a Whitney twist of a graph $G$ produces a (generally) non-isomorphic graph whose cycle matroid is isomorphic to the cycle matroid of $G$, a twisted flip of a biased graph $\GB$ produces a (generally) non-isomorphic biased graph $\GBp$ with $F\GBp \iso F\GB$.  
This operation is described toward the end of Section \ref{sec:Preliminaries}.

\subsection{Standard notions: biased graphs, minors, connectivity} 
\label{sec:standardnotionsbiasedgraphsminorsconnectivity} 

For a frame matroid $M$ represented by a biased graph $\Om = \GB$, we denote throughout by $E = E(M) = E(G)$ the common ground set of $M$ and edge set of $G$.  
When it is important to distinguish an edge which is not a loop from one that is, we refer to an edge having distinct endpoints as a \emph{link}.  
There are minor operations we may perform on $\GB$ that correspond to minor operations in $M$, as follows \cite{MR1088626}. 
For an element $e \in E$, \emph{delete} $e$ from $\GB$ by deleting $e$ from $G$ and removing from $\Bb$ every cycle containing $e$.  
To \emph{contract} $e$, there are three cases: 
If $e$ is a balanced loop, $\GB /e = \GB \setminus e$.  
If $e$ is a link, contract $e$ in $G$ and declare a cycle $C$ to be balanced if either $C \in \Bb$ or $E(C) \cup \{e\}$ forms a cycle in $\Bb$.  
If $e$ is an unbalanced loop with endpoint $u$, then $\GB /e$ 
is the biased graph obtained from \GB{} as follows: $e$ is deleted, all other loops incident to $u$ become balanced, and links incident to $u$ become unbalanced loops incident to their other endpoint.  
A \emph{minor} of $\GB$ is any biased graph obtained by a sequence of deletions and contractions.  
It is readily checked that these minor operations on biased graphs preserve the theta property, and that they agree with matroid minor operations on their frame matroids; that is, for any element $e \in E$, $F\GB \setminus e = F((G, \Bb) \setminus e)$ and $F\GB /e = F((G, \Bb) /e)$ (this shows that the class of frame matroids is minor closed).  

For a biased graph $\Omega \eq (G, \mathcal{B})$ we say $G$ is the \emph{underlying graph} of $\Omega$.  
We write $\Om[X]$ or $G[X]$ to denote the biased subgraph of $\GB$ induced by the edges in $X$ that has balanced cycles just those cycles in $\Bb$ whose edge set is contained in $X$.  
If $G[X]$ contains no unbalanced cycle, it is \emph{balanced}; otherwise it is \emph{unbalanced}.  
If $G[X]$ contains no balanced cycle, it is \emph{contrabalanced}.  
We denote by $V(X)$ the set of vertices incident with an edge in $X$, and by $b(X)$ the number of balanced components of $G[X]$.  
It follows from the definitions that for a frame matroid $M$ represented by biased graph $\GB$, the rank of $X$ in $M$ is $r(X) = |V(X)| - b(X)$.

A \emph{separation} of a graph $G{=}(V,E)$ is a pair of edge disjoint subgraphs $G_1, G_2$ of $G$ with $G = G_1 \cup G_2$.  
The \emph{order} of a separation is $|V(G_1) \cap V(G_2)|$.  
A separation of order $k$ is a \emph{$k$-separation}.  
If both $V(G_1) \setminus V(G_2)$ and $V(G_2) \setminus V(G_1)$ are non-empty, then the separation is \emph{proper}.  
If $G$ has no proper separation of order less than $k$, then $G$ is \emph{$k$-connected}.  
The least integer $k$ for which $G$ has a proper $k$-separation is the \emph{connectivity} of $G$.  
A partition $(X,Y)$ of $E$ naturally induces a separation $G[X], G[Y]$ of $G$, which we also denote $(X,Y)$.  
We call $X$ and $Y$ the \emph{sides} of the separation.  
The \emph{connectivity function} of $G$ is the function $\lambda_G$ that to each partition $(X,Y)$ of $E$ assigns the order of its corresponding separation; that is, $\lambda_G(X,Y) = |V(X) \cap V(Y)|$.  

A \emph{$k$-separation} of a biased graph $\Om{=}\GB$ is a $k$-separation of its underlying graph $G$, and the \emph{connectivity} of $\Om$ is that of $G$.  
The \emph{connectivity function} $\lambda_\Om$ of $\Om$ is that of $G$.  

A \emph{separation} of a matroid $M$ is a partition of its ground set $E$ into two subsets $X$, $Y$; it is also denoted $(X,Y)$, with $X$ and $Y$ the \emph{sides} of the separation.  
The \emph{order} of a separation $(X,Y)$ of a matroid is $r(X) + r(Y) - r(E) + 1$.  
A separation of order $k$ with both $|X|, |Y| \geq k$ is a \emph{$k$-separation}.  
If $M$ has no $l$-separation with $l<k$, then $M$ is \emph{$k$-connected.}  
The \emph{connectivity} of $M$ is the least integer $k$ such that $M$ has a $k$-separation, provided one exists (otherwise we say the connectivity of $M$ is infinite).  
Evidently, $M$ is connected \iiff $M$ has no 1-separation.  
The \emph{connectivity function} of a matroid $M$ on ground set $E$ is the function $\lambda_M$ that assigns to each separation $(X,Y)$ of $E$ its order; that is, $\lambda_M(X,Y) = r(X) + r(Y) - r(M) + 1$.  

Let $M$ be a frame matroid represented by a biased graph $\Om$.  
The following facts regarding the relationship between the order of a separation $(X,Y)$ in $M$ and the order of $(X,Y)$ in $\Om$ will be used extensively throughout.  
In general, a separation has different orders in $\Om$ and $F(\Om)$.  
However, if the sides of a separation are connected in the graph then this difference is at most one.  
To see this, let $(X,Y)$ be a partition of $E$.  
The order of $(X,Y)$ in $M$ is 
\begin{equation}  \label{eqn:matroid_graph_conn}
\begin{aligned} 
\lambda_M(X,Y) &= r(X) + r(Y) - r(M) + 1 \\ &= |V(X)| - b(X) + |V(Y)| - b(Y) - (|V| - b(E)) + 1 \\
&= |V(X) \cap V(Y)| - b(X) - b(Y) + b(E) + 1 \\ &= \lambda_\Om(X,Y) - b(X) - b(Y) + b(E) + 1. 
\end{aligned}  
\end{equation}
Suppose both $\Om[X]$ and $\Om[Y]$ connected.  
If $\Om$ is balanced, we have $\lambda_M(X,Y) = \lambda_\Omega(X,Y)$.  
If $\Om$ is unbalanced, we have 
\begin{enumerate} 
\item  if both $\Om[X]$ and $\Om[Y]$ are unbalanced, $\lambda_M(X,Y) = \lambda_\Om(X,Y)+1$,
\item  if one of $\Om[X]$ or $\Om[Y]$ is balanced while the other is unbalanced, then $\lambda_M(X,Y) = \lambda_\Om(X,Y)$, and 
\item  if both $\Om[X]$ and $\Om[Y]$ are balanced, then 
$\lambda_M(X,Y) = \lambda_\Om(X,Y)-1$.
\end{enumerate} 

Moreover, (\ref{eqn:matroid_graph_conn}) immediately implies that if $M$ is connected, then $\Om$ must be connected: if there is a partition $(X,Y)$ of $E$ with $\lambda_\Om(X,Y)=0$, then $\lambda_M(X,Y) = 1$.   
The converse need not hold: a frame matroid represented by a connected biased graph may be disconnected.  
Indeed, let $M{=}F(\Om)$, where $\Om$ is connected, and suppose $(X,Y)$ is a 1-separation of $M$, with both $\Om[X]$ and $\Om[Y]$ connected.  
If $\Om$ is balanced, then $\lambda_M(X,Y)=\lambda_\Om(X,Y)=1$.  
Otherwise, by (\ref{eqn:matroid_graph_conn}) one of the following holds: 
\begin{itemize} 
\item  $\lambda_\Om(X,Y)=1$, and precisely one of $\Om[X]$ or $\Om[Y]$ is balanced; 
\item  $\lambda_\Om(X,Y)=2$, and each of $\Om[X]$ and $\Om[Y]$ are balanced.    
\end{itemize}

Throughout, matroids and biased graphs are finite; graphs may have loops and parallel edges.  
We often make no distinction between a subset of elements $A$ of a matroid $M = F\GB$, the subset of edges of $G$ representing an edge in $A$, and the biased subgraph $G[A]$ induced by $A$.

\subsection{Excluded minors are connected, simple, and cosimple} 

Having established the standard vocabulary of biased graphs and connectivity, we may immediately make the observations that an excluded minor is connected, simple, and cosimple.  

\begin{obs} \label{obs:an_ex_min_is_connected}
If $M$ is an excluded minor for the class of frame matroids, then $M$ is connected.  
\end{obs} 

We denote the direct sum of two matroids $M$ and $N$ by $M \oplus N$.  
Evidently, if $\Omega$ and $\Psi$ are biased graphs, then the disjoint union $\Om \dcup \Psi$ of $\Om$ and $\Psi$ represents $F(\Om) \oplus F(\Psi)$.  
We denote the restriction of a matroid $M$ to a subset $A \subseteq E(M)$ by $M|A$.  
If $M = F(\Om)$, then clearly $\Om[A]$ is a biased graph representing $M|A$.

\begin{proof}[Proof of Observation \ref{obs:an_ex_min_is_connected}]
Suppose to the contrary that $M$ is an excluded minor, and that $M$ has a 1-separation $(A,B)$.  
Then $M$ is the direct sum of its restrictions to each of $A$ and $B$.  
By minimality, each of $M|A$ and $M|B$ are frame.  
Let $\Omega$ and $\Psi$ be biased graphs representing $M|A$ and $M|B$ respectively, and let $\Omega \mathbin{\dot{\cup}} \Psi$ denote the biased graph which is the disjoint union of $\Omega$ and $\Psi$.  
Then $M = M|A \onesum M|B = F(\Omega) \onesum F(\Psi) = F(\Omega \mathbin{\dot\cup} \Psi)$, so $M$ is frame, a contradiction.  
\end{proof} 

\begin{obs} \label{obs:ex_min_is_simple_and_cosimple}
Let $M$ be an excluded minor for the class of frame matroids.  
Then $M$ is simple and cosimple.  
\end{obs} 

\begin{proof} 
Suppose $M$ has a loop $e$.  
By minimality, there is a biased graph $\GB$ representing $M \setminus e$.  Adding a balanced loop labelled $e$ incident to any vertex of $G$ yields a biased graph representing $M$, a contradiction.  
Similarly, if $M$ has a coloop $f$, consider a biased graph $\GB$ representing $M/f$.  
Adding a new vertex $w$, choosing any vertex $v \in V(G)$, and adding edge $f= vw$ to $G$ yields a biased graph representing $M$, a contradiction.  

Now suppose $M$ has a two-element circuit $\{e,f\}$.  
Let $\GB$ be a biased graph representing $M \setminus e$.  
If $f$ is a link in $G$, say $f = uv$, then let $G'$ be the graph obtained from $G$ by adding $e$ in parallel with $f$ so $e$ also has endpoints $u$ and $v$, and let $\Bb' = \Bb \cup \{C \setminus f \mathbin{\cup} e \mid f \subset C \in \Bb\}$.  
If $f$ is an unbalanced loop in $G$, say incident to $u \in V(G)$, then let $G'$ be the graph obtained from $G$ by adding $e$ as an unbalanced loop also incident with $u$, and let $\Bb' = \Bb$.  
Then $M = F(G',\Bb')$, a contradiction.  

Similarly, if $e$ and $f$ are elements in series in $M$, let $\GB$ be a biased graph representing $M/e$.  
If $f$ is a link in $G$, say $f=uv$, then let $G'$ be the graph obtained from $G$ by deleting $f$, adding a new vertex $w$, and putting $f=uw$ and $e=wv$; let $\Bb' = \{C \mid C \in \Bb$ or $C/e \in \Bb\}$.  
If $f$ is an unbalanced loop in $G$, say incident to $u \in V(G)$, let $G'$ be the graph obtained from $G$ by deleting $f$, adding a new vertex $w$, and adding edges $e$ and $f$ in parallel, both with endpoints $u,w$; let $\Bb'=\Bb$ (so $\{e,f\}$ is an unbalanced cycle).   
Again, then $M = F(G',\Bb')$, a contradiction.  
\end{proof}

\subsection{Working with biased graphs} 
\label{sec:Preliminaries}  

Before determining further properties of excluded minors, we need to develop some tools and establish some basic facts about biased graphs.  
If $X, Y$ are subgraphs of a graph $G$, an \emph{$X$-$Y$ path} in $G$ is a path that meets $X \cup Y$ exactly in its endpoints, with one endpoint in $X$ and the other in $Y$.

\paragraph{Rerouting.}

Let $G$ be a graph, let $P$ be a path in $G$, and let $Q$ be a path internally disjoint from $P$ linking two vertices $x, y \in V(P)$.  
We say the path $P'$ obtained from $P$ by replacing the subpath of $P$ linking $x$ and $y$ with $Q$ is obtained by \emph{rerouting} $P$ along $Q$.  

\begin{obs} 
\label{state:rerouting_paths}
Given two $u$-$v$ paths $P, P'$ in a graph, $P$ may be transformed into $P'$ by a sequence of reroutings.  
\end{obs}

\begin{proof}
To see this, suppose $P$ and $P'$ agree on an initial segment from $u$.  
Let $x$ be the final vertex on this common initial subpath.  
If $x=v$, then $P=P'$, so assume $x \not= v$.  
Let $y$ be the next vertex of $P'$ following $x$ that is also in $P$.  
Denote the subpath of $P'$ from $x$ to $y$ by $Q$.  
Since $y$ is different from $x$, the path obtained by rerouting $P$ along $Q$ has a strictly longer common initial segment with $P'$ than $P$.  
Continuing in this manner, eventually $x = v$, and $P$ has been transformed into $P'$.  
\end{proof}

The relevance of this for us is the following simple fact.  
If a subpath $R$ of a path $P$ is rerouted along $Q$, and the cycle $R \cup Q$ is balanced, we refer to this as rerouting \emph{along a balanced cycle}.  
If $C$ is a cycle, $x, y$ are distinct vertices in $C$, $P$ is an $x$-$y$ path contained in $C$, $Q$ is an $x$-$y$ path internally disjoint from $C$, and the cycle $P \cup Q$ is balanced, we say the cycle $C'$ obtained from $C$ by rerouting $P$ along $Q$ is obtained from $C$ by \emph{rerouting along a balanced cycle}.  
The following fact will be used extensively.  

\begin{lem} \label{lem:rerouting_along_a_bal_cycle}
If $C'$ is obtained from $C$ by rerouting along a balanced cycle, then $C$ and $C'$ have the same bias.  
\end{lem}

\begin{proof} 
Since $C \cup Q$ is a theta subgraph, this follows immediately from the theta property.  
\end{proof}

\paragraph{Signed graphs.} 
A \emph{signed} graph consists of a graph $G$ together with a distinguished subset of edges $\Sigma \subseteq E(G)$ called its \emph{signature}.  
A signed graph naturally gives rise to a biased graph $(G,\Bb_{\Sigma})$ in which a cycle $C \in \Bb_{\Sigma}$ \iiff{} $|E(C) \cap \Sigma|$ is even (it is immediate that $\Bb_{\Sigma}$ satisfies the theta property).   
We say that an arbitrary biased graph $(G,\Bb)$ \emph{is} a signed graph if there exists a set $\Sigma \subseteq E(G)$ so that $\Bb_{\Sigma} = \Bb$.  The following gives a characterisation of when this occurs.  

\begin{prop} \label{prop:If_no_odd_theta} 
A biased graph is a signed graph \iiff it contains no contrabalanced theta subgraph.  
\end{prop}

\begin{proof} 
First suppose that $(G,\Bb)$ is a signed graph, and choose $\Sigma \subseteq E(G)$ so that $\Bb_{\Sigma} = \Bb$.  If $P_1, P_2, P_3$ are three internally disjoint paths forming a theta subgraph in $G$ then two of $|E(P_1) \cap \Sigma|$, $|E(P_2) \cap \Sigma|$, and $|E(P_3) \cap \Sigma|$ have the same parity, and these paths will form a balanced cycle.  Thus, every theta subgraph contains a balanced cycle, and thus no contrabalanced theta subgraph exists.  

To prove the converse, let $(G,\Bb)$ be a biased graph which has no contrabalanced theta subgraph.  We may assume $G$ is connected; if not, apply the following argument to each component of $G$.  Let $T$ be a spanning tree of $G$ and define $\Sigma$ by the following rule: 
\[ \Sigma = \{ e \in E(G) \setminus E(T) \mid \mbox{the unique cycle in $T + e$ is unbalanced} \}. \]
We claim that $\Bb_{\Sigma} = \Bb$.  To prove this, we will show that a cycle $C$ is in $\Bb$ \iiff $|E(C) \cap \Sigma|$ is even, and we do this by induction on the number of edges in $C \setminus E(T)$.  

If all but one edge $e$ of $C$ is contained in $T$, then the result holds by definition of $\Sigma$.  
Suppose $|E(C) \setminus E(T)|=n \geq 2$, and the result holds for all cycles having less than $n$ edges not in $T$.  
Choose a minimal path $P$ in $T \setminus E(C)$ linking two vertices $x, y$ in $V(C)$ (such a path exists since $C$ has at least two edges not in $T$: say $e=uv, f \in C \setminus T$; the $u$-$v$ path in $T$ avoids $f$ and so at some vertex leaves $C$ and then at some vertex returns to $C$).  
Cycle $C$ is the union of two internally disjoint $x$-$y$ paths $P_1, P_2$ and together $P, P_1, P_2$ form a theta subgraph of $G$.  
Let $C_1 = P_1 \cup P$ and $C_2 = P_2 \cup P$.  
Since $(G,\Bb)$ has no contrabalanced theta, the cycle $C$ is unbalanced if and only if exactly one of $C_1$ and $C_2$ is unbalanced.  However, 
by induction (each of $C_1$ and $C_2$ has fewer edges not in $T$), this holds if and only if $|E(C_1) \cap \Sigma|$ and 
$|E(C_2) \cap \Sigma|$ have different parity.  This is equivalent to $|E(C) \cap \Sigma|$ being odd, thus completing the argument.
\end{proof}

\paragraph{Balancing vertices.} 
If $\Omega = (G,\Bb)$ is a biased graph and $v \in V(G)$ we let $\Omega - v$ denote the biased graph $(G - v, \Bb')$ where $\Bb'$ consists of all cycles in $\Bb$ which do not contain $v$.  
A vertex $v$ in a biased graph $\Omega$ is \emph{balancing} if $\Omega - v$ is balanced.  
When $\Omega$ has a balancing vertex, the biases of its cycles have a simple structure.  
We denote the set of links incident with a vertex $v$ in a graph by $\delta(v)$.  

\begin{obs} \label{state:ei_ej_with_balancing_vertex_have_same_parity_cycles}
Let $\GB$ be a biased graph and suppose $u$ is a balancing vertex in $\GB$.  
Let $\delta(u) = \{e_1, \ldots, e_k\}$.  
For each pair of edges $e_i, e_j$ ($1 \leq i < j \leq k$), either all cycles containing $e_i$ and $e_j$ are balanced or all cycles containing $e_i$ and $e_j$ are unbalanced.  
\end{obs}

\begin{proof} 
Fix $i, j$, and consider two cycles $C$ and $C'$ containing $e_i$ and $e_j$.  
Let $e_i = ux_i$ and $e_j = ux_j$.  
Write $C = u e_i x_i P x_j e_j u$ and $C' = u e_i x_i P' x_j e_j u$.  
Path $P$ may be transformed into $P'$ by a sequence of reroutings, $P{=}P_0, P_1, \ldots, P_l{=}P'$ in $G-u$.  
Since $u$ is balancing, each rerouting is along a balanced cycle.  
Hence by Lemma \ref{lem:rerouting_along_a_bal_cycle}, at each step $m \in \{1,\ldots, l\}$, the cycles $u e_i x_i P_{m-1} x_j e_j u$ and $u e_i x_i P_m x_j e_j u$ have the same bias.  
\end{proof} 

The above fact prompts the introduction of a relation on $\delta(v)$ for a balancing vertex $v$.  Namely, we define $\sim$ on $\delta(v)$ by the rule that $e,f \in \delta(v)$ satisfy $e \sim f$ if either $e = f$ or there exists a balanced cycle containing both $e$ and $f$.  Clearly $\sim$ is reflexive and symmetric.  The relation $\sim$ is also transitive: 
Suppose $e_1 \sim e_2$ and $e_2 \sim e_3$ and let $e_i{=}v x_i$ for $1 \le i \le 3$.  Since there is a balanced cycle containing $x_1 v x_2$ and a balanced cycle containing $x_2 v x_3$, there is an $x_1$-$x_2$ path avoiding $v$ and an $x_2$-$x_3$ path avoiding $v$.  
Hence there is an $x_1$-$x_3$ path $P$ avoiding $v$ and a $P$-$x_2$ path $Q$ avoiding $v$.  
Let $P \cap Q = \{y\}$.  Together, $v$, $e_1$, $e_2$, $e_3$, $P$, and $Q$ form a theta subgraph of $G$.  
By Observation \ref{state:ei_ej_with_balancing_vertex_have_same_parity_cycles}, the cycle of this theta containing $e_1,e_2$ and the cycle containing $e_2,e_3$ are both balanced.  
It follows that the cycle of this theta containing $e_1,e_3$ is also balanced, so $e_1 \sim e_3$.  We summarize this important property below.  

\begin{obs}
If $v$ is a balancing vertex of $\Omega$, there exists an equivalence relation $\sim$ on $\delta(v)$ so that a cycle $C$ of $\Omega$ containing $v$ is balanced if and only if it contains two edges from the same equivalence class.  
\end{obs} 
\noindent We call the $\sim$ classes of $\delta(v)$ its \emph{b-classes}.  

\paragraph{$k$-signed graphs.}  
These are generalisations of signed graphs which we use to work with biased graphs with balancing vertices and other related biased graphs.  
A \emph{$k$-signed graph} consists of a graph $G$ together with a collection $\mathbf{\Sigma} = \{ \Sigma_1, \ldots, \Sigma_k \}$ of subsets of $E(G)$, which we again call its \emph{signature}.  
A $k$-signed graph gives rise to a biased graph $(G, \Bb_{\mathbf{\Sigma}})$ in which a cycle $C \in \Bb_{\mathbf{\Sigma}}$ \iiff $|E(C) \cap \Sigma_i|$ is even for every $1 \le i \le k$.  
Again it is straightforward to verify that $\mathbf{\Sigma}$ satisfies the theta property.  
We say that an arbitrary biased graph $(G,\Bb)$ \emph{is} a $k$-signed graph if there exists a collection $\mathbf{\Sigma}$ so that $\Bb_{\mathbf{\Sigma}} = \Bb$.  
A 1-signed graph is a signed graph.  
The reader familiar with group-labelled graphs will note that signed graphs are group-labelled graphs where the associated group is $\mathbb{Z}_2{=}\mathbb{Z}/2\mathbb{Z}$, and our $k$-signed graphs are group-labelled by $\mathbb{Z}_2^k$. 

\begin{obs} \label{obs:relabelling_when_there_is_a_balancing_vertex}
Let $(G,\Bb)$ be a biased graph with a balancing vertex $v$ after deleting its set $U$ of unbalanced loops.  
Let $\{ \Sigma_1, \ldots, \Sigma_k \}$ be the partition of $\delta(v)$ into b-classes in $\GB \setminus U$, and let $\mathbf{\Sigma} = \{ U, \Sigma_1, \ldots, \Sigma_k \}$.  
Then $(G,\Bb)$ is a $k$-signed graph with $\Bb_{\mathbf{\Sigma}} = \Bb_{\mathbf{\Sigma} \setminus \Sigma_i} = \Bb$ for every $1 \le i \le k$.  
\end{obs}

\begin{proof} 
This follows easily from the fact that $\sim$ is an equivalence relation in $\GB \setminus U$.  
\end{proof}

\paragraph{Biased graph representations.}
In general, a frame matroid $M$ has more than one biased graph representing $M$.  
We will encounter several situations in which non-isomorphic biased graphs represent the same frame matroid.  
For our purposes, we require three results on non-isomorphic biased graphs representing the same frame matroid.  

(1) \ 
If $H$ is a graph, one way to obtain an unbalanced biased graph $\GB$ with $F\GB \iso M(H)$ is to \emph{pinch} two vertices of $H$, as follows.  
Choose two distinct vertices $u, v \in V(H)$, and let $G$ be the graph obtained from $H$ by identifying $u$ and $v$ to a single vertex $w$.  
Then $\delta(w) = \delta(u) \cup \delta(v) \setminus \{e \mid e = uv\}$ (an edge with endpoints $u$ and $v$ becomes a loop incident to $w$).  
Let $\Bb$ be the set of all cycles in $G$ not meeting both $\delta(u)$ and $\delta(v)$.  
It is easy to see that the circuits of the two matroids agree, so $F\GB \iso M(H)$.  
The biased graph $\GB$ obtained by pinching $u$ and $v$ of $H$ is a signed graph with $\Sigma = \delta(u)$

The signed graph obtained by pinching two vertices of a graph has a balancing vertex.  
Conversely, if $\GB$ is a signed graph with a balancing vertex $u$, then $\GB$ is obtained as a pinch of a graph $H$, which we may describe as follows.  
If $|\delta(u)/{\sim}| > 2$, then $u$ is a cut vertex and each block of $G$ contains at most two b-classes, else $\GB$ contains a contrabalanced theta, contradicting Proposition \ref{prop:If_no_odd_theta}.  
Hence each block of $G$ contains edges in at most two b-classes of $\delta(u)$.  
By Observation \ref{obs:relabelling_when_there_is_a_balancing_vertex} then, $\Bb$ is realised in each block $G_i$ ($i \in \{1, \ldots, n\}$) of $G$ by a single set $\Sigma_i  \subseteq \delta(u)$.  
Clearly, taking $\Sigma = \bigcup_i \Sigma_i$ yields $\Bb_\Sigma = \Bb$.  
Let $H$ be the graph obtained from $G$ by \emph{splitting} vertex $u$;  that is, replace $u$ with two vertices, $u'$ and $u''$, put all edges in $\Sigma$ incident to $u'$, and all edges in $\delta(u) \setminus \Sigma$ incident to $u''$; put unbalanced loops as $u'u''$ edges and leave balanced loops as balanced loop incident to either $u'$ or $u''$.  It is easily verified that $M(H) \iso F\GB$: 

\begin{prop} \label{prop:vertex_splitting_operation}
Let $(G, \Bb)$ be a signed graph with a balancing vertex $u$.  
If $H$ is obtained from $(G, \Bb)$ by splitting $u$, then $M(H) \iso F(G, \Bb)$.  
\end{prop} 

(2) \ 
If $\GB$ is a biased graph with a balancing vertex $u$, then the following \emph{roll-up} operation produces another biased graph with frame matroid isomorphic to $F\GB$.  
Let $\Sigma_j = \{e_1, \ldots, e_l\}$ be the set of edges of one of the b-classes in $\delta(u)$.  
Let $(G',\Bb')$ be the biased graph obtained from $\GB$ by replacing each edge $e_i = u v_i  \in \Sigma_j$ with an unbalanced loop incident to its endpoint $v_i$ (see Figure \ref{fig:Rep_with_loops_bal_vertex}).  
\begin{figure}[tbp] 
\begin{center} 
\includegraphics[scale=0.9]{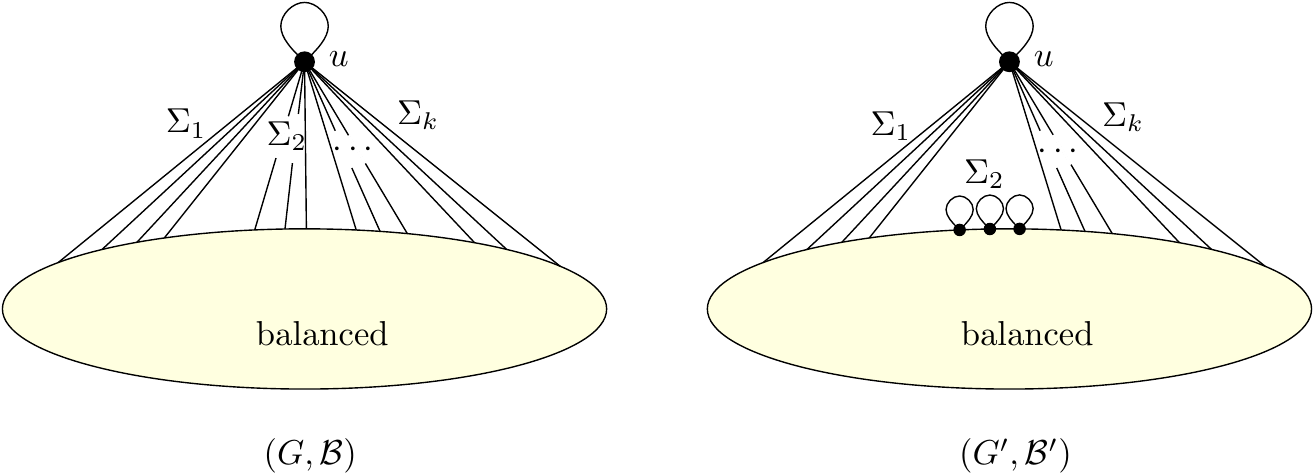}
\end{center} 
\caption{$F\GB \iso F\GBp$}
\label{fig:Rep_with_loops_bal_vertex} 
\end{figure} 
It is straightforward to check that $F\GB$ and $F\GBp$ have the same set of circuits.  
We say biased graph $\GBp$ is obtained from $\GB$ by the \emph{roll-up} of b-class $\Sigma_j$ of $\delta(u)$.  
This operation may also be performed in reverse: 
Let $\GBp$ be a biased graph in which $u$ is a balancing vertex after deleting unbalanced loops, and let $\Sigma_1, \ldots, \Sigma_k$ be the b-classes of $\delta(u)$ after deleting all unbalanced loops.  
Let $\GB$ be the biased graph obtained from $\GBp$ by replacing each unbalanced loop with an edge between $u$ and its original end, putting $\Sigma_{k+1} = \{e \mid e$ is an unbalanced loop in $\GBp\}$, setting $\mathbf{\Sigma} = \{\Sigma_1, \ldots, \Sigma_{k+1}\}$, and taking $\Bb=\Bb_{\mathbf{\Sigma}}$.  
Then $F\GB \iso F\GBp$, and we say $\GB$ is obtained from $\GBp$ by \emph{unrolling} the unbalanced loops of $\GBp$.  
Hence if $M$ is a frame matroid represented by a biased graph with a balancing vertex $u$ after deleting unbalanced loops, then every biased graph obtained by unrolling unbalanced loops, then rolling up a b-class of $\delta(u)$, also represents $M$.  
Observe also that if $H$ is a graph and $v \in V(H)$, then the biased graph $\GB$ obtained by rolling up all edges in $\delta(v)$ has $F\GB \iso M(H)$.  

(3) \ 
The operations of pinching and splitting, and of rolling up the edges of a b-class or unrolling unbalanced loops, are all special cases of the following \emph{twisted flip} operation.  
This operation may be applied to $k$-signed graphs having the following structure.  
Let $G$ be a graph, let $u \in V(G)$, let $G_0, \ldots, G_m$ be edge disjoint connected subgraphs of $G$, and let $\mathbf{\Sigma} = \{ \Sigma_1, \ldots, \Sigma_k \}$ be a collection of subsets of $E(G)$ satisfying the following (see Figure \ref{fig:flips_more_general}(a)).  
\begin{enumerate}

\item $E(G) \setminus \bigcup_{i=0}^m E(G_i)$ is empty or consists of loops at $u$.

\item $E(G_0) \cap \Sigma_i = \emptyset$ for $1 \le i \le k$.

\item  For every $1 \le i \le m$ there is a vertex $x_i$ so that $V(G_i) \cap \br{\bigcup_{j \not= i} V(G_j)} \subseteq \{u, x_i\}$.  

\item For every $1 \le i \le m$ there exists a unique $s_i$, $1 \le s_i \le k$, so that $E(G_i) \cap \Sigma_j = \emptyset$ for $j \neq s_i$.

\item Every edge in $E(G_i) \cap \Sigma_{s_i}$ is incident with $x_i$.
\end{enumerate}

Consider the resulting biased graph $(G, \Bb_{\mathbf{\Sigma}})$ and its associated frame matroid $F(G, \Bb_{\mathbf{\Sigma}})$.  
We obtain a biased graph $(G', \Bb_{\mathbf{\Sigma'}})$ with $F(G', \Bb_{\mathbf{\Sigma'}}) \iso F(G, \Bb_{\mathbf{\Sigma}})$ from $(G, \Bb_{\mathbf{\Sigma}})$ as follows (see Figures \ref{fig:flips_more_general}(b) and \ref{fig:flip_singlelobe}).  

\begin{itemize}
\item Redefine the endpoints of each edge of the form $e {=} y u \notin \Sigma_{s_i}$ so that $e {=} y x_i$ (note that an edge $e{=}x_i u \notin \Sigma_{s_i}$ thus becomes a loop $e{=}x_i x_i$).  

\item Redefine the endpoints of each edge of the form $e{=}y x_i \in \Sigma_{s_i}$ with $y \neq u$ so that $e{=}y u$.  

\item  For each $1 \leq j \le k$, let $\Sigma_j' = \{e \mid $ the endpoints of $e$ have been redefined so that $e{=}yx_i$ for some $y \in V(G_i) \} \cup \{ e \mid e{=}x_i u \in \Sigma_j\}$.  
Put $\mathbf{\Sigma'} = \{\Sigma_1', \ldots, \Sigma_k'\}$.  
\end{itemize}

\begin{thm} \label{prop:flip_operation}
If $(G', \Bb_{\mathbf{\Sigma'}})$ is obtained from $(G, \Bb_{\mathbf{\Sigma}})$ as a twisted flip, then $F(G, \Bb_{\mathbf{\Sigma}}) \iso F(G', \Bb_{\mathbf{\Sigma'}})$.  
\end{thm} 

\begin{proof} 
It is straightforward to check that $F(G, \Bb_{\mathbf{\Sigma}})$ and $F(G', \Bb_{\mathbf{\Sigma'}})$ have the same set of circuits.  
\end{proof} 

\begin{figure}[tbp] 
\begin{center} 
\includegraphics[scale=0.8]{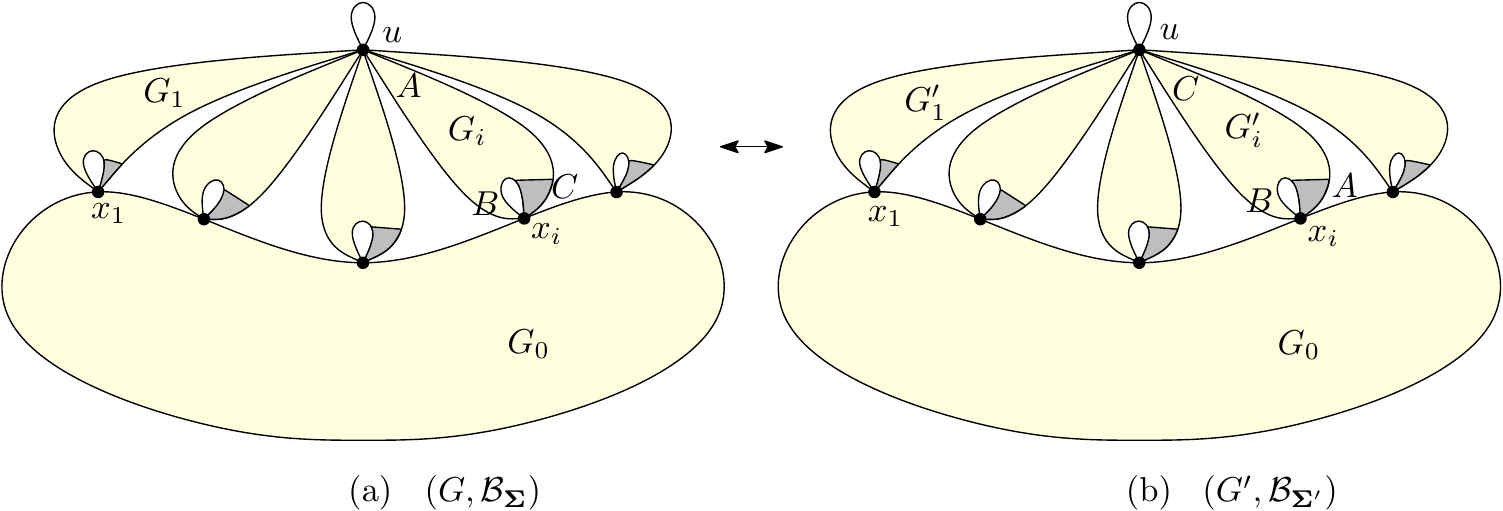}
\end{center} 
\caption{A twisted flip: Edges in $\mathbf{\Sigma}$ and $\mathbf{\Sigma'}$ are shaded; edges marked $A$ in $G$ become incident to $x_i$ in $G'$ and are in $\mathbf{\Sigma'}$; edges marked $C$ in $G$ become incident to $u$ in $G'$.}
\label{fig:flips_more_general} 
\end{figure} 

\begin{figure}[tbp] 
\begin{center} 
\includegraphics[scale=0.8]{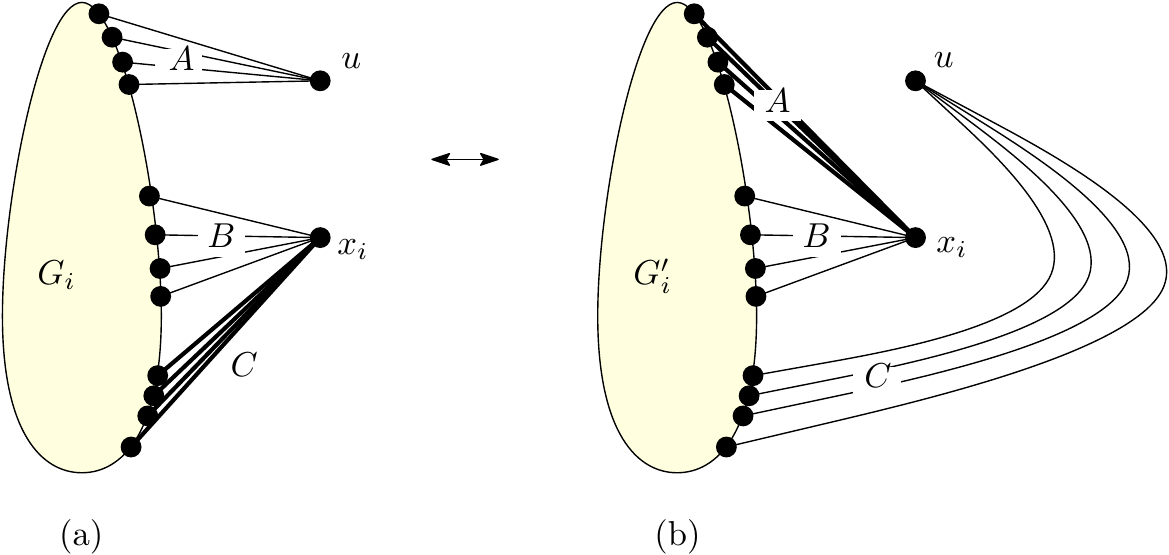}
\end{center} 
\caption{A twisted flip's effect on a single biased subgraph $G_i$.  Edges contained in $\mathbf{\Sigma}$ and in $\mathbf{\Sigma}'$ are bold.  In $G_i$ (a) edges marked $C$ are in some $\Sigma_i$, and in $G_i'$ (b) edges marked $A$ are then in $\Sigma_i'$.}
\label{fig:flip_singlelobe} 
\end{figure}

Observe that if $u$ is a balancing vertex in a biased graph $\GB$, and $A \in \delta(u)/\sim$, then applying Observation \ref{obs:relabelling_when_there_is_a_balancing_vertex} yields a signature $\mathbf{\Sigma} \subseteq E(G)$ so $\Bb = \Bb_{\mathbf{\Sigma}}$, with the property that $A$ is disjoint from the members of $\mathbf{\Sigma}$.  
Then a twisted flip operation on $(G,\Bb_{\mathbf{\Sigma}})$ is the operation of rolling up b-class $A$.  
A pinch operation is obtained as a twisted flip by taking $G = G_1$ and $\mathbf{\Sigma} = \emptyset$; then $(G,\Bb)$ is balanced, and the biased graph $(G',\Bb')$ given by a twisted flip is that obtained by pinching the vertices $v$ and $x_1$.  
Additionally, the special case when each $\delta_{G_i}(x_i) \cap \Sigma_{s_i} = \emptyset$ and there is no unbalanced loop incident to $u$ is the curling operation in \cite{MR3417216}.  

\section{2-sums of frame matroids and matroidals}
\label{sec:2-sums of frame matroids}

In this section we provide necessary and sufficient conditions for a 2-sum of two frame matroids to be frame, Theorem \ref{thm:biased2sum} below.  

The \emph{2-sum} of two matroids $M_1$ and $M_2$ \emph{on} elements $e_1 \in E(M_1)$ and $e_2 \in E(M_2)$, denoted $M_1 \twosume{e_1}{e_2} M_2$, is the matroid on ground set $\br{ E(M_1) \cup E(M_2) } \setminus \{e_1,e_2\}$ with circuits: the circuits of $M_i$ avoiding $e_i$ for $i=1,2$, together with $\{ \br{C_1 \cup C_2} \setminus \{e_1,e_2\} \mid C_i$ is a circuit of $M_i$ containing $e_i$ for $i=1,2 \}$.  
The following result (independently of Bixby, Cunningham, and Seymour) is fundamental.  

\begin{thm}[\cite{oxley:mt},Theorem 8.3.1] \label{thm:2sep_iff_2sum}
A connected matroid $M$ is not 3-connected \iiff there are matroids $M_1$, $M_2$, each of which is a proper minor of $M$, such that $M$ is a 2-sum of $M_1$ and $M_2$.  
\end{thm} 

If $M$ is a matroid whose automorphism group is transitive on $E(M)$, then we write simply $M \twosume{}{f} N$ to indicate the 2-sum of $M$ and $N$ taken on some element $e \in E(M)$ and element $f \in E(N)$; if also $N$ has transitive automorphism group we may simply write $M \twosum N$.

\paragraph{Matroidals.} 
A \emph{matroidal} is a pair $(M,L)$ consisting of a matroid $M$ together with a distinguished subset $L$ of its elements.  
A matroidal $\Mm{=}(M,L)$ is \emph{frame} if there is a biased graph $\Om$ with $M=F(\Om)$ in which every element in $L$ is an unbalanced loop.  
We say a biased graph in which all elements in $L \subseteq E(\Om)$ are unbalanced loops is \emph{$L$-biased}.  
Thus $\Mm{=}(M,L)$ is a frame matroidal \iiff there exists an $L$-biased graph $\Om$ with $F(\Om)=M$.  
In this case we say $\Om$ \emph{represents} $\Mm$.  

Observe that, as long as $M$ is simple, this is equivalent to asking that there be a frame for $M$ containing $L$.  
To see this, recall the construction given on page \pageref{pagerefforbiaedgraphconstruction} of a biased graph representing a matroid $M$ with frame $B$.  
Though it is not required that the frame $B$ be disjoint from $E$, the construction assumes $B \cap E = \emptyset$.  
We can do away with this assumption as follows.  
Suppose $B \cap E = F$.  
Construct $\GB$ with edge set $E \setminus F$ as before.  
Now add an unbalanced loop incident to each vertex of $G$, and let each element of the frame be represented by the new loop incident to its vertex.  
Thus we obtain a frame extension $N$ without any added parallel elements in which all elements in the frame are unbalanced loops in the biased graph representing $N$.  
Conversely, as long as $M$ is simple, given a biased graph $\Om$ representing $M$, the set of unbalanced loops of $\Om$ is contained in a frame for $M$ | namely, after adding an unbalanced loop at each vertex not already having one, the basis consisting of the set of unbalanced loops.    

The main result of this section says that a 2-sum of two non-graphic frame matroids is frame \iiff each of the summands has a frame containing the element upon which the 2-sum is taken.  

\begin{thm} \label{thm:biased2sum}
Let $M_1,M_2$ be connected matroids and for $i=1,2$ let $e_i \in E(M_i)$.  The matroid $M_1 \twosume{e_1}{e_2} M_2$ is frame if and only if one of the following holds.
\begin{enumerate}
\item One of $M_1$ or $M_2$ is graphic and the other is frame.
\item  Both matroidals $(M_1,\{e_1\})$ and $(M_2,\{e_2\})$ are frame.  
\end{enumerate}
\end{thm}

We prove a more general statement than Theorem \ref{thm:biased2sum}, giving necessary and sufficient conditions for a 2-sum of two frame matroidals to be frame.  
This more general result will be required in Section \ref{sec:Excluded_minors}.  
The statement and its proof will be given after the following necessary preliminaries.

\subsection{2-summing biased graphs}

Let $\Omega_1, \Omega_2$ be biased graphs and let $e_i \in E(\Omega_i)$ for $i=1,2$.  
There are two ways in which we may perform a biased graphical 2-sum operation on $\Om_1$ and $\Om_2$ to obtain a biased graph representing the 2-sum $F(\Om_1) \twosume{e_1}{e_2} F(\Om_2)$.  

\begin{enumerate}
\item Suppose that $e_i$ is an unbalanced loop in $\Om_i$ incident with vertex $v_i$, for $i \in \{1, 2\}$.  
The \emph{loop-sum} of $\Omega_1$ and $\Omega_2$ \emph{on} $e_1$ and $e_2$ is the biased graph obtained from the disjoint union of $\Omega_1 - e_1$ and $\Omega_2 - e_2$ by identifying vertices $v_1$ and $v_2$.  
Every cycle in the loop-sum is contained in one of $\Om_1$ or $\Om_2$; its bias is defined accordingly.  

\item Suppose that $\Om_1$ is balanced, and that $e_i$ is a link in $\Omega_i$ incident with vertices $u_i, v_i$, for $i \in \{1,2\}$.  
The \emph{link-sum} of $\Omega_1$ and $\Omega_2$ \emph{on} $e_1$ and $e_2$ is the biased graph obtained from the disjoint union of $\Omega_1 - e_1$ and $\Omega_2 - e_2$ by identifying $u_1$ with $u_2$ and $v_1$ with $v_2$.  
A cycle in the link-sum is balanced if it is either a balanced cycle in $\Omega_1$ or $\Omega_2$ or if it may be written as a union $(C_1 \setminus e_1) \cup (C_2 \setminus e_2)$ where for $i \in \{1,2\}$, $C_i$ is a balanced cycle in $\Omega_i$ containing $e_i$.   
(It is straightforward to verify that the theta rule is satisfied by this construction.)  
\end{enumerate}

\begin{prop} \label{prop:biasedgraph_twosums}
Let $\Omega_1, \Omega_2$ be biased graphs and let $e_i \in E(\Omega_i)$ for $i \in \{1,2\}$.  
If $\Omega$ is a loop-sum or link-sum of $\Omega_1$ and $\Omega_2$ on $e_1$ and $e_2$, then $F(\Omega) = F(\Omega_2) \twosume{e_1}{e_2} F(\Omega_2)$.
\end{prop}

\begin{proof}
It is easily checked that for both the loop-sum and link-sum, the circuits of $F(\Omega)$ and of $F(\Om_1) \twosume{e_1}{e_2} F(\Om_2)$ coincide, regardless of the choice of pairs of endpoints of $e_1$ and $e_2$ that are identified in the link-sum.  
\end{proof}

\subsection{Decomposing along a 2-separation} 

By Theorem \ref{thm:2sep_iff_2sum}, a matroid $M$ of connectivity 2 decomposes into two of its proper minors such that $M$ is a 2-sum of these smaller matroids.  
If $M$ is frame, then every minor of $M$ is frame, and we would like to be able to express the 2-sum in terms of a loop-sum or link-sum of two biased graphs representing these minors.  
This motivates the following definitions.  
Let $M$ be a connected frame matroid on $E$ and let $\Omega$ be a biased graph representing $M$.   
A 2-separation $(A,B)$ of $M$ is a \emph{biseparation of} $\Omega$.  
There are four types of biseparations that play key roles.  
Define a biseparation to be type 1, 2(a), 2(b), 3(a), 3(b), or 4, respectively, if it appears as in Figure \ref{fig:types_of_biseparations}, where each component of $\Om[A]$ and $\Om[B]$ is connected; components of each side of the separation marked ``b'' are balanced, those marked ``u'' are unbalanced.  
We refer to a biseparation of type 2(a) or 2(b) as {type 2}, and a biseparation of type 3(a) or 3(b) as {type 3}.  

\begin{figure}[tbp] 
\begin{center} 
\includegraphics[scale=0.8]{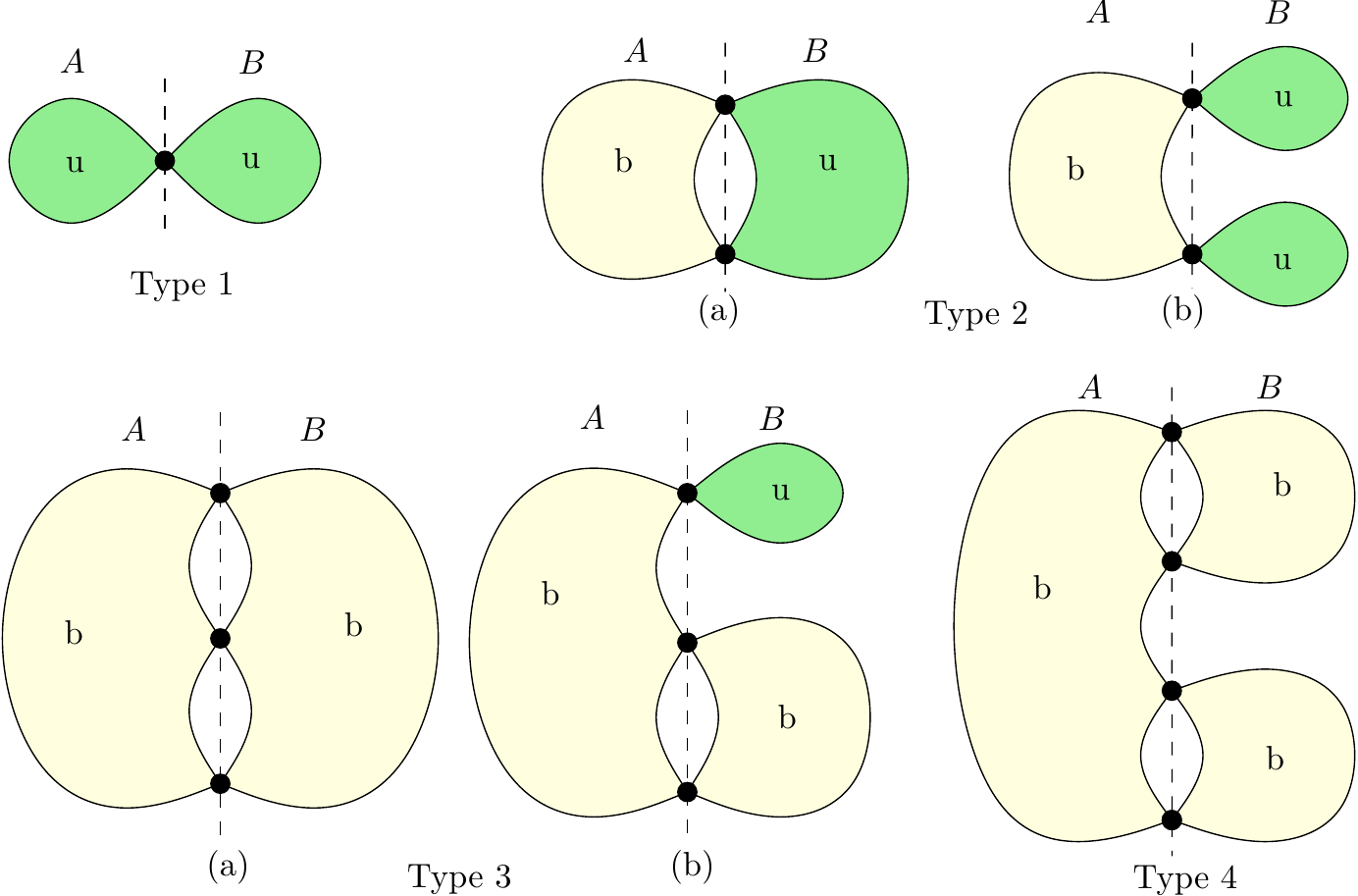}
\end{center} 
\caption{Four types of biseparations.}
\label{fig:types_of_biseparations}
\end{figure} 

\begin{prop} \label{prop:type1or2_bisep_is_link_loop_sum} 
Let $M$ be a connected frame matroid such that $M = M_1 \twosume{e_1}{e_2} M_2$ for two matroids $M_1, M_2$.  
Let $\Omega$ be a biased graph representing $M$, and let $E(M_i) \setminus \{e_i\} = E_i$ for $i \in \{1,2\}$.  
If $(E_1, E_2)$ is type 1 (resp.\ type 2), then there exist biased graphs $\Omega_i$ with $E(\Omega_i) = E(M_i)$, $i\in\{1,2\}$, such that $\Omega$ is the loop-sum (resp.\ link-sum) of $\Omega_1$ and $\Omega_2$ on $e_1$ and $e_2$.  
\end{prop}

\begin{proof} 
If $(E_1, E_2)$ is type 1, then for $i \in \{1,2\}$ let $\Om_i$ be the biased graph obtained from $\Om$ by replacing $\Om[E_{i+1}]$ with an unbalanced loop $e_i$ incident to the vertex in $V(E_1) \cap V(E_2)$ (adding indices modulo 2).  
Then $\Om$ is the loop-sum of $\Om_1$ and $\Om_2$ on $e_1$ and $e_2$.  

Now suppose $(E_1, E_2)$ is type 2.  
Let $V(E_1) \cap V(E_2) = \{x,y\}$, and assume without loss of generality that $\Om[E_1]$ is balanced while $\Om[E_2]$ is unbalanced.  
For $i \in \{1, 2\}$ let $\Om_i$ be the biased graph obtained from $\Om$ by replacing $\Om[E_{i+1}]$ with a link $e_i = xy$ and defining bias as follows: $\Om_1$ is balanced, while 
the balanced cycles of $\Om_2$ are precisely those not containing $e_2$ that are balanced in $\Om$ together with those cycles containing $e_2$ for which replacing $e_2$ with an $x$-$y$ path in $\Om[E_1]$ yields a balanced cycle in $\Om$ (Lemma \ref{lem:rerouting_along_a_bal_cycle} guarantees that this collection is well-defined).  
Then $\Om$ is the link-sum of $\Om_1$ and $\Om_2$ on $e_1$ and $e_2$.  
\end{proof}

\subsubsection{Taming biseparations}

In light of Proposition \ref{prop:type1or2_bisep_is_link_loop_sum}, we want to show that for every 2-separation of a frame matroid $M$, there exists a biased graph representing $M$ for which the corresponding biseparation is type 1 or 2.  
We first show that there is always such a representation in which the biseparation is type 1, 2, or 3.    
In preparation for the more general form of Theorem \ref{thm:biased2sum} we wish to prove, we now consider matroidals.  
We say a matroidal $\Mm{=}(M,L)$ is connected if $M$ is connected.  

\begin{lem} \label{lem:tamebisep}
Let $\Mm{=}(M,L)$ be a connected frame matroidal.  
For every 2-separation $(A,B)$ of $M$, there exists an $L$-biased representation of $\Mm$ for which $(A,B)$ is type 1, 2, or 3.
\end{lem}

\begin{proof}
Choose an $L$-biased representation $\Omega$ of $(M,L)$ for which $\Omega$ is not balanced (any balanced representation can be turned into an unbalanced one by a pinch or roll-up operation, so this is always possible).  
Let $S = V(A) \cap V(B)$ in $\Om$.  
Let $\{A_1, \ldots, A_h\}$ be the partition of $A$ and $\{B_1, \ldots, B_k\}$ the partition of $B$ so that every $\Omega[A_i]$ is a component of the biased graph $\Omega[A]$ and every $\Omega[B_j]$ is a component of the biased graph $\Omega[B]$.  
Call the graphs $\Omega[A_1], \ldots, \Omega[A_h], \Omega[B_1], \ldots, \Omega[B_k]$ \emph{parts}.   
For every $1 \le i \le h$ (resp.\ $1 \le j \le k$) let $\delta_A^i = 1$ ($\delta_B^j=1$) if $\Omega[A_i]$ is balanced ($\Omega[B_j]$ is balanced) and $\delta_A^i = 0$ ($\delta_B^j = 0$) otherwise.  
Then $\lambda_M(A,B) = 2 = 1 + |S| - \sum_{i=1}^h \delta_A^i - \sum_{j=1}^k \delta_B^j$.  
Since each vertex in $S$ is in exactly one $\Omega[A_i]$ and exactly one $\Omega[B_j]$, doubling both sides of this equation and rearranging, we  obtain 
\[ 2 = \sum_{i=1}^h \br{ |S \cap V(A_i)| - 2\delta_A^i } + \sum_{j=1}^k \br{ |S \cap V(B_j)| - 2\delta_B^j } . \]
If a part is balanced, it must contain at least two vertices in $S$ (else $M$ is not connected by the discussion in Section \ref{sec:standardnotionsbiasedgraphsminorsconnectivity}), so every term in the sums on the right hand side of the above equation is nonnegative.  
In particular, letting $t$ be the number of vertices in $S$ contained in a part, a balanced part will contribute $t-2$ to the sum, and an unbalanced part will contribute $t$.  
Call a part \emph{neutral} if it is balanced and contains exactly two vertices in $S$.  
Since the total sum is two, the possibilities for the parts of $\Om[A]$ and $\Om[B]$ are: 
\begin{enumerate} 
\item[(a)]  two unbalanced parts each with one vertex in $S$ and all other parts neutral, 
\item[(b)]  one unbalanced part with two vertices in $S$ and all other parts neutral, 
\item[(c)]  one balanced part with three vertices in $S$, one unbalanced part with one vertex in $S$, and all other parts neutral, 
\item[(d)]  two balanced parts with three vertices in $S$ and all other parts neutral, or 
\item[(e)]  one balanced part with four vertices in $S$ and all other parts neutral.  
\end{enumerate}  
These possibilities are illustrated in Figure \ref{fig:structures_of_biseparations}.  

\begin{figure}[tbp] 
\begin{center} 
\includegraphics[scale=0.8]{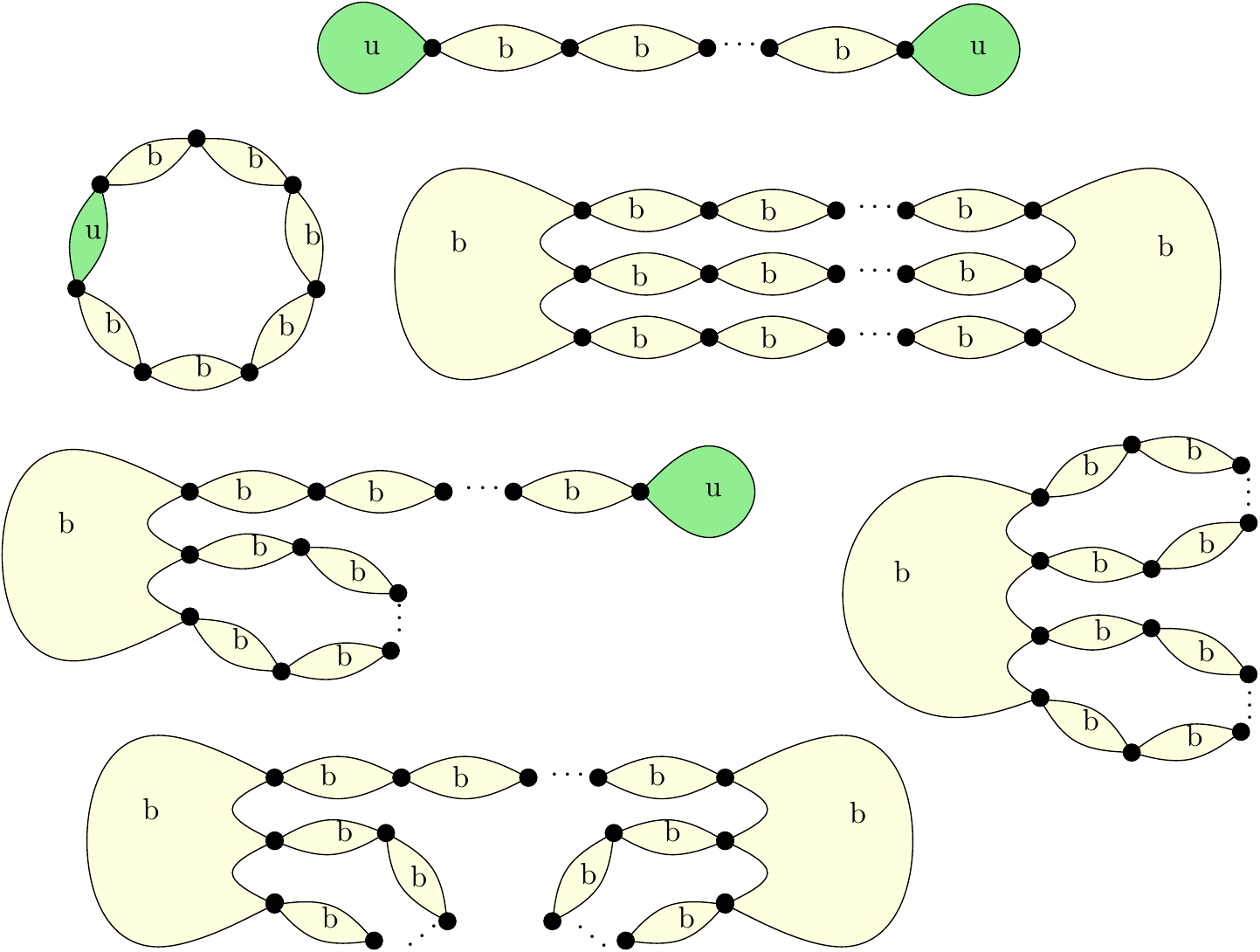}
\end{center} 
\caption[Possible decompositions of $\Om$ into the parts]{The possible decompositions of $\Om$ into the parts of $\Om[A]$ and $\Om[B]$.  } 
\label{fig:structures_of_biseparations}
\end{figure} 

Observe that every component of $M \setminus B$ (resp.\ $M \setminus A$) is contained in some part $\Om[A_i]$ ($\Om[B_j]$), and every part of $\Om[A]$ (resp.\ $\Om[B]$) is a union of components of $M \setminus B$ (resp.\ $M \setminus A$).  
Hence every circuit of $M$ is either contained in a single part, or traverses every part.  
It is an elementary property of 2-separations that if $A_1, \ldots, A_l$ and $B_1, \ldots, B_m$ are the components of $M \setminus B$ and $M \setminus A$ respectively, and $(X,Y)$ is any partition of $A_1, \ldots, A_l, B_1, \ldots, B_m$, then $(\bigcup X, \bigcup Y)$ is a 2-separation of $M$ (this can be verified by straightforward rank calculations).  
Hence if $\Om[D]$ is a neutral part, $(D, D^c)$ is a 2-separation of $M$.  Since $\Om[D]$ is balanced and connected, the biseparation $(D,D^c)$ of $\Om$ is type 2.  

Suppose there are exactly $t$ neutral parts.  
Repeatedly applying Proposition \ref{prop:type1or2_bisep_is_link_loop_sum}, we obtain a biased graph $\Om'$ with links $e_1', \ldots, e_t'$ together with balanced biased graphs $\Om_1, \ldots, \Om_t$ each with a distinguished edge $e_i \in E(\Omega_i)$ so that $\Om$ is obtained as a repeated link-sum of $\Om'$ with each $\Om_i$ on edges $e_i$ and $e_i'$.  
It follows from the fact that every circuit of $M$ is either contained in a single part or traverses every part that elements $e_1', \ldots, e_t'$ are all in series in $F(\Om')$.  
We use this fact to find another biased graph representing $M$ in which the biseparation $(A,B)$ is type 1, 2, or 3.  
First, in $\Om'$ contract edges $e_1', \ldots, e_{t-1}'$: let $\Om'' = \Om' / \{e_1', \ldots, e_{t-1}'\}$.  
Now subdivide link $e_t'$ to form a path $P$ with edge set $\{e_1', \ldots, e_t'\}$ to obtain a new biased graph $\Psi$, in which a cycle containing $P$ is balanced \iiff the corresponding cycle in $\Om''$ containing $e_t'$ is balanced.  
Since elements $e_1', \ldots, e_t'$ are in series in $F(\Om')$, $F(\Psi) \iso F(\Om')$.  
For the same reason, any biased graph $\Psi'$ obtained from $\Psi$ by permuting the order in which edges $e_1', \ldots, e_t'$ occur in $P$ has $F(\Psi') \iso F(\Psi)$.  
Let $\Phi'$ be the biased graph obtained from $\Psi$ by arranging the edges of $P$ in an order so that an initial segment of the path has all of the edges $e_i'$ whose corresponding neutral parts of $\Om$ are in $A$, followed by the edges $e_i'$ whose corresponding neutral parts are in $B$.  
Now let $\Phi$ be the biased graph obtained by repeatedly link-summing each $\Om_i$ on edge $e_i'$, $i \in \{1, \ldots, t\}$.  
Then $F(\Phi) \iso F(\Om)$.  
Since every unbalanced loop in $\Om$ remains an unbalanced loop in $\Phi$, $\Phi$ is an $L$-biased representation of $\Mm$.  
Since at least one of $\Phi[A]$ or $\Phi[B]$ is connected, and $\Phi[A]$ and $\Phi[B]$ meet in at most three vertices, in $\Phi$ biseparation $(A,B)$ is type 1, 2, or 3.  
\end{proof}

\subsubsection{Taming type 3}

We now do away with type 3 biseparations.  

\begin{thm} \label{thm:tamebisep}
Let $\Mm{=}(M,L)$ be a connected frame matroidal.  
For every 2-separation $(A,B)$ of $M$, there exists an $L$-biased representation of $\Mm$ for which $(A,B)$ is type 1 or 2.
\end{thm}

\begin{proof}
By Lemma \ref{lem:tamebisep} we may choose an $L$-biased graph $\Omega$ representing $M$ in which biseparation $(A, B)$ is type 1, 2, or 3.  
Suppose it is type 3.  
Let $\{x,y,z\} = V(A) \cap V(B)$ in $\Om$.  
We first claim that all cycles crossing $(A,B)$ through the same pair of vertices $\{x,y\}$, $\{y,z\}$, or $\{z,x\}$ have the same bias.  To see this, let $C$ and $C'$ be two cycles crossing $(A,B)$ at $\{x,y\}$.  We may assume without loss of generality that $\delta(z) \cap C \subseteq A$.  Let $C \cap A = P$ and $C \cap B = Q$, and 
let $C' \cap A = P'$ and $C' \cap B = Q'$.  
By Observation \ref{state:rerouting_paths}, $P$ may be transformed to $P'$ by a sequence of reroutings in $P \cup P'$.  
Since every rerouting in this sequence is along a balanced cycle, by Lemma \ref{lem:rerouting_along_a_bal_cycle}, $C$ and $P' \cup Q$ have the same bias.  
Similarly, $Q$ may be transformed into $Q'$ via a sequence of reroutings along balanced cycles in $Q \cup Q'$, so $P' \cup Q$ and $P' \cup Q' = C'$ have the same bias.  \Ie, $C$ and $C'$ are of the same bias.  

There are  three types of cycles crossing the 2-separation: those crossing at $\{x,y\}$, those crossing at $\{x,z\}$, and those crossing at $\{y,z\}$; by the claim, all cycles of the same type have the same bias.  
Let us denote the sets of these cycles by $\Cc_{xy}$, $\Cc_{xz}$ and $\Cc_{yz}$, respectively.  

We claim that at least one of these sets contains an unbalanced cycle.  
For suppose the contrary.  
If the biseparation of $\Om$ is type 3(a), then $\Om$ is balanced with $|V(A) \cap V(B)|=3$; but then $(A,B)$ is not a 2-separation of $F(\Om)$, a contradiction.  
If the biseparation is type 3(b), then $M$ is not connected, a contradiction.   

Suppose first that just one of our sets of cycles, say $\Cc_{xy}$, contains an unbalanced cycle $C$.  
Suppose further that in one of $\Om[A]$ or $\Om[B]$ there is a $z$-$C$ path $P$ avoiding $x$ and that in the other side there is a $z$-$C$ path $Q$ avoiding $y$.   
Then $C \cup P \cup Q$ is a theta subgraph of $\Om$ containing exactly two balanced cycles, a contradiction.  
So no such pair of paths exist.  
Hence either: 
\begin{enumerate} 
\item  both $\Om[A]$ and $\Om[B]$ contain a $z$-$C$ path, but either every $z$-$C$ path in both meets $x$ or every $z$-$C$ path in both meets $y$, or, 
\item  one of $\Om[A]$ or $\Om[B]$ has no $z$-$C$ path.  
\end{enumerate}
In case 1, either $x$ or $y$ is a cut vertex of $\Om$, and we find that $F(\Om)$ is not connected, a contradiction.  
Hence we have case 2.  Suppose \wolog{} that $\Om[B]$ does not contain a $z$-$C$ path.  
We have a biseparation of type 3(b).  
Let us denote by $B_1$ the balanced component and by $B_2$ the unbalanced component of biased graph $\Om[B]$.  
Let $\Phi$ be the biased graph obtained as follows.  
Detach $B_2$ from $\Om[A]$, and form a signed graph $(G, \Bb_\Sigma)$ from $\Om[A]$ by identifying vertices $x$ and $y$, and setting $\Sigma = \delta(y) \cap A$.  
Now identify vertex $x$ in $B_1$ with vertex $z$ in $(G, \Bb_\Sigma)$, and identity vertex $y$ in $B_1$ with vertex $z$ in $B_2$ (Figure \ref{fig:doawaywithtype31}).  
Assign biases to cycles in $\Phi$ in $\Phi[A]$ according to their bias in $(G, \Bb_\Sigma)$ and in $\Phi[B]$ according to their bias in $\Om$.  
It is straightforward to verify that the circuits of $F(\Phi)$ and $F(\Om)$ coincide, so $F(\Phi) \iso M$.  
The biseparation $(A, B)$ in $\Phi$ is type 1, and since edges representing elements in $L$ remain unbalanced loops in $\Phi$, $\Phi$ is an $L$-bias representation of $M$ as required.  

\begin{figure}[tbp] 
\begin{center} 
\includegraphics[scale=0.8]{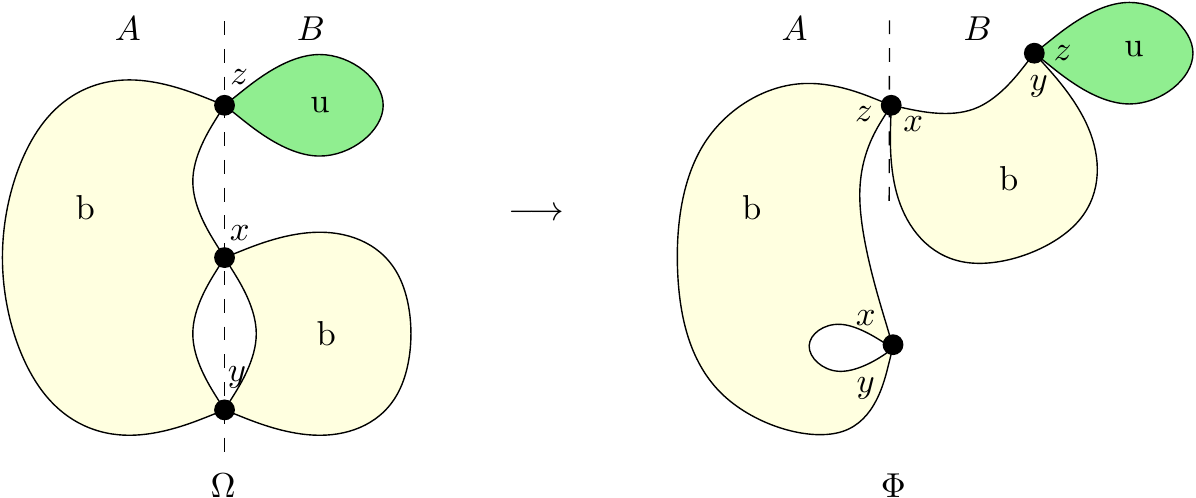}
\end{center} 
\caption{Finding a representation in which the biseparation if type 1.}
\label{fig:doawaywithtype31} 
\end{figure} 

So now assume that at least two of the three sets $\Cc_{xy}$, $\Cc_{yz}$ and $\Cc_{xz}$ contain an unbalanced cycle.  
Then our biseparation is type 3(a).  
If just two of these sets contain an unbalanced cycle | say $\Cc_{xz}$ does not | then $M$ is graphic, represented by the graph obtained from $\Om$ by splitting vertex $y$ (Figure \ref{fig:two_of_Cc_xy_Cc_yz_unbalanced}).  
Now pinching vertices $x$ and $z$ yields an $L$-biased graph representing $M$ in which biseparation $(A,B)$ is type 1.  

\begin{figure}[tbp]
\begin{center} 
\includegraphics[scale=0.8]{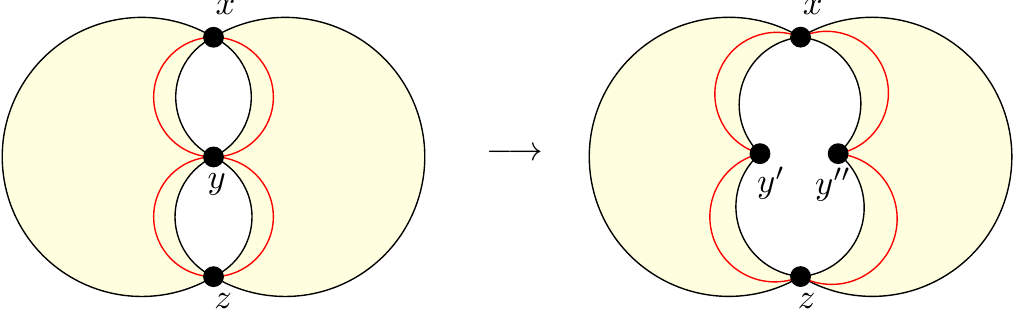} 
\end{center} 
\caption[If just $\Cc_{xy}$ and $\Cc_{yz}$ contain unbalanced cycles]{If just $\Cc_{xy}$ and $\Cc_{yz}$ contain unbalanced cycles, then $F(\Om)$ is graphic.}  
\label{fig:two_of_Cc_xy_Cc_yz_unbalanced} 
\end{figure} 

The remaining case is that all three of $\Cc_{xy}, \Cc_{xz}$, and $\Cc_{yz}$ contain unbalanced cycles, so every cycle crossing $(A,B)$ is unbalanced.  
In this case every circuit of $M$ contained in $A$ or $B$ is a balanced cycle and every circuit meeting both $A$ and $B$ is either a pair of tight handcuffs meeting at a vertex in $V(A) \cap V(B)$, or an contrabalanced theta (Figure \ref{fig:type3_to_type1_2separation}).  
\begin{figure}[tbp]
\begin{center}
\includegraphics[scale=0.8]{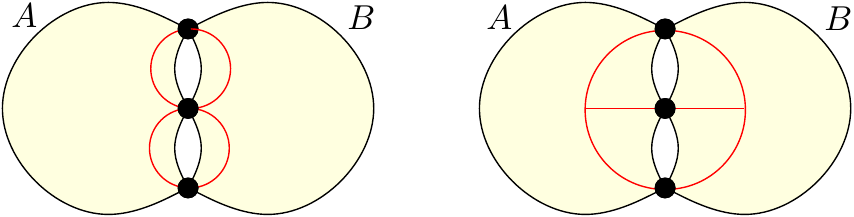}
\end{center}
\caption{Circuits of $F(\Om)$ meeting both sides of the 2-separation.}
\label{fig:type3_to_type1_2separation}
\end{figure}
Let $\Omega'$ be the signed graph obtained from $\Omega$ as follows.  
Split vertices $y$ and $z$, replacing $y$ with two new vertices $y'$ and $y''$, putting all edges $uy \in A$ incident with $y'$ and all edges $vy \in B$ incident with $y''$ and similarly replacing $z$ with two new vertices $z'$ and $z''$, putting all edges $uz \in A$ incident with $z'$ and all edges $vz \in B$ incident with $z''$.   
Now identify vertices $y'$ and $z'$ and identify vertices $y''$ and $z''$, and put the edges in $\delta(z) \cap A$ and in $\delta(z) \cap B$ in $\Sigma$ (Figure \ref{fig:type3_to_type1_2separation2}).  
It is easily checked that a subset $C \subseteq E$ is a circuit in $F(\Om)$ \iiff $C$ is a circuit in $F(\Om')$, so $F(\Om') \iso F(\Om)$.  
Since in this case $L$ is empty, $\Om'$ is an $L$-biased graph representing $M$, as required.  
\begin{figure}[tbp]
\begin{center}
\includegraphics[scale=0.8]{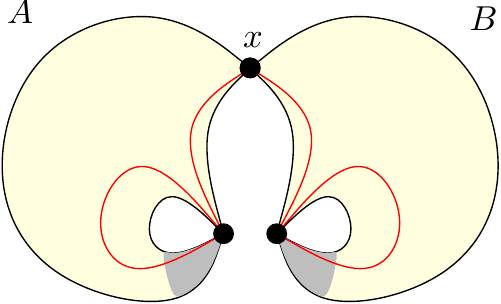}
\end{center}
\caption{A representation in which the biseparation is type 1}
\label{fig:type3_to_type1_2separation2}
\end{figure}
\end{proof}

\subsection{Proof of Theorem \ref{thm:biased2sum}}

With this we are ready to prove the main result of this section.  

\begin{lem} \label{lem:workhorse2sum}
Let $\Mm_1{=}(M_1,L_1)$ and $\Mm_2{=}(M_2, L_2)$ be connected frame matroidals on $E_1, E_2$, respectively.  
If for $i=1, 2$, $e_i \in E_i \setminus L_i$, then $(M_1 \twosume{e_1}{e_2} M_2, L_1 \cup L_2)$ is frame if and only if one of the following holds.  
\begin{enumerate}
\item $L_i = \emptyset$ and $M_i$ is graphic for one of $i=1$ or $i=2$.
\item $(M_i, L_i \cup \{e_i\})$ is frame for both $i=1, 2$.  
\end{enumerate}
\end{lem}

\begin{proof}
The ``if'' direction follows immediately from Proposition \ref{prop:biasedgraph_twosums}.  
Conversely, consider a frame matroidal resulting from a 2-sum, $(M_1 \twosume{e_1}{e_2} M_2, L_1 \cup L_2)$.  
By Theorem \ref{thm:tamebisep} there is a $(L_1 \cup L_2)$-biased representation $\Om$ of the 2-sum in which the biseparation $(E_1 \setminus e_1, E_2 \setminus e_2)$ is type 1 or 2.  
By Proposition \ref{prop:type1or2_bisep_is_link_loop_sum}, there are biased graphs $\Om_1$ on $E_1$ and $\Om_2$ on $E_2$ such that $\Om$ is a link- or loop-sum on $e_1$ and $e_2$.  
If $\Om$ is a link-sum, then 1 holds.  
If $\Om$ is a loop-sum, then both $\Om_i$ are $(L_i \cup \{e_i\})$-biased representations of $M_i$, so both matroidals $(M_i, L_i \cup \{e_i\})$ are frame ($i \in \{1,2\}$).  
\end{proof} 

Lemma \ref{lem:workhorse2sum} immediately implies Theorem \ref{thm:biased2sum}.  

\begin{proof}[Proof of Theorem \ref{thm:biased2sum}]
Apply Lemma \ref{lem:workhorse2sum} with $L_1 = L_2 = \emptyset$.
\end{proof}

\section{Excluded minors}
\label{sec:Excluded_minors} 

In this section we use Theorem \ref{thm:biased2sum} to construct a family $\Ee_0$ of 9 excluded minors with connectivity 2.  
We then show that any excluded minor of connectivity 2 that is not in $\Ee_0$ has a special structure.  

\subsection{The excluded minors $\Ee_0$}

The graph obtained from $K_{3,3}$ by adding an edge $e'$ linking two non-adjacent vertices is denoted $K_{3,3}'$; we also denote the corresponding element of $M^*(K_{3,3}')$ by $e'$.  
Let 
\begin{align*} 
\Ee_0 =& \{ U_{2,4} \twosum M^*(H) \mid H \in \{K_5, K_{3,3}, K_{3,3}'\}\} \\ 
	&\cup \{ M^*(H_1) \twosum M^*(H_2) \mid H_1, H_2 \in \{K_5, K_{3,3}, K_{3,3}'\} \}, 
\end{align*}
where the 2-sum is taken on $e'$ whenever $H$, $H_1$ or $H_2$ is $K_{3,3}'$.  

There are three biased graphs representing $U_{2,4}$, two biased graphs representing $M^*(K_5)$, and just one biased graph representation of $M^*(K_{3,3})$ \cite{MR1058551}.  
These are shown in Figure 
\ref{fig:biased_graphs_of_ex_min_for_graphic}.  
\begin{figure}[tbp]
\begin{center}
\includegraphics[scale=0.8]{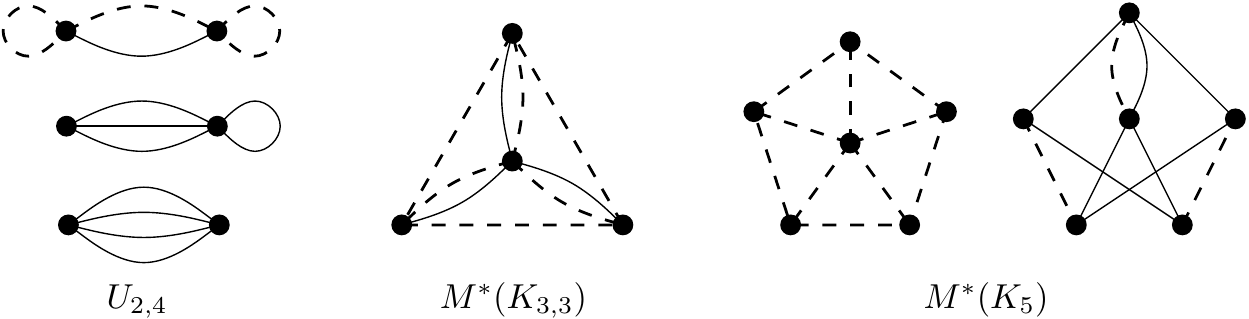}
\caption{The biased graphs representing excluded minors for the class of graphic matroids.  Those with dashed edges are signed graphs, with signatures given by dashed edges.  The other two biased graphs, representing $U_{2,4}$, have all cycles unbalanced.} 
\label{fig:biased_graphs_of_ex_min_for_graphic}
\end{center}
\end{figure}
There are two biased graphs representing $M^*(K_{3,3}')$, shown in Figure \ref{fig:K_33_+_an_edge}.  
\begin{figure}[tbp]
\begin{center}
\includegraphics[scale=0.8]{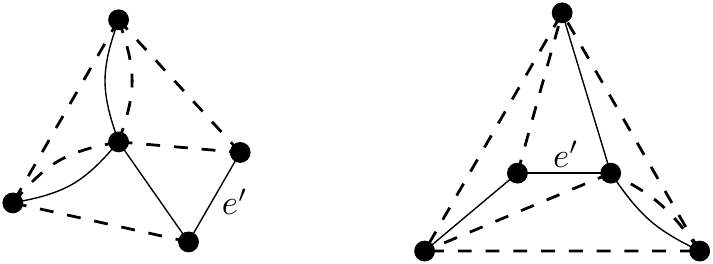}
\end{center}
\caption[$M^*(K_{3,3}')$]{The biased graphs representing $M^*(K_{3,3}')$; both are signed graphs with signature the dashed edges.} 
\label{fig:K_33_+_an_edge}
\end{figure}

%

\begin{thm} \label{thm:e0excl}
The matroids in $\Ee_0$ are excluded minors for the class of frame matroids.  
\end{thm}

\begin{proof}
Let $M_1 \twosum M_2 \in \Ee_0$, with $M_1$ one of $U_{2,4}$, $M^*(K_5)$, $M^*(K_{3,3})$, or $M^*(K_{3,3}')$ and $M_2$ one of $M^*(K_5),$ $M^*(K_{3,3})$, or $M^*(K_{3,3}')$.  
Since neither $M_1$ nor $M_2$ is graphic and $M_2$ has no representation with a loop, by Theorem \ref{thm:biased2sum} $M_1 \twosum M_2$ is not frame.  
Since every proper minor of $U_{2,4}$, $M^*(K_5)$, and $M^*(K_{3,3})$ is graphic, and for every $e \not=e'$, both $M^*(K_{3,3}') \setminus e$ and $M^*(K_{3,3}') /e$ are graphic, 
every proper minor of $M_1 \twosum M_2$ is a 2-sum of a graphic matroid and a frame matroid.  
Hence by Theorem \ref{thm:biased2sum}, every proper minor of $M_1 \twosum M_2$ is frame.  
\end{proof}

\subsection{Other excluded minors of connectivity 2}

We now investigate excluded minors of connectivity 2 that are not in $\Ee_0$.  
We show that any such excluded minor has the following structure.  
For a matroid $M$ and subset $L \subseteq E(M)$, the matroid obtained by taking a 2-sum of a copy of $U_{2,4}$ on each element in $L$ is denoted $M \gltwosum{L}$.  

\begin{thm} \label{thm:exclminr2sumu24}
Let $M$ be an excluded minor for the class of frame matroids.  If $M$ has connectivity 2 and is not in $\Ee_0$, then $M =  N \gltwosum{L}$ for a 3-connected frame matroid $N$.  
\end{thm}

We prove Theorem \ref{thm:exclminr2sumu24} via three lemmas, each of which requires some explanation.  

A collection $\Nn$ of connected matroids is \emph{1-rounded} if it has the property that whenever a connected matroid $M$ has a minor $N \in \Nn$, then every element $e \in E(M)$ is contained in some minor $N'$ of $M$ with $N' \in \Nn$.  
The following is a result of Seymour (\cite{oxley:mt} Section 11.3). 

\begin{thm}
\label{thm:Seymour_a_1_rounded_family}
The collection $\{U_{2,4},$ $F_7, F_7^*,$ $M^*(K_5),$ $M^*(K_{3,3}),$ $M^*(K_{3,3}')\}$ is 1-rounded.    
\end{thm} 

We use Theorem \ref{thm:Seymour_a_1_rounded_family} in the proof of the following lemma, to find a minor containing the base point on which a 2-sum is taken.  

\begin{lem} \label{lem:none0struc}
Let $M_1,M_2$ be nontrivial matroids and suppose $M_1 \twosume{e_1}{e_2} M_2$ is an excluded minor for the class of frame matroids, for some $e_1 \in E(M_1)$ and $e_2 \in E(M_2)$.  Then either $M_1 \twosume{e_1}{e_2} M_2 \in \Ee_0$ or both $M_1$ and $M_2$ are non-binary frame matroids.   
\end{lem} 

\begin{proof}
By minimality, $M_1$ and $M_2$ are both frame.  
By Theorem \ref{thm:biased2sum}, neither $M_1$ nor $M_2$ is graphic.  
Thus each contains an excluded minor for the class of graphic matroids, namely, one of $U_{2,4}$, $F_7$, $F_7^*$, $M^*(K_5)$, or $M^*(K_{3,3})$.  
By Theorem \ref{thm:Seymour_a_1_rounded_family}, for $i \in \{1,2\}$, matroid $M_i$ contains a minor $N_i$ isomorphic to one of $U_{2,4}$, $F_7$, $F_7^*$, $M^*(K_5)$, $M^*(K_{3,3})$, or  $M^*(K_{3,3}')$ with $e_i \in E(N_i)$; we may assume that if $N_i \iso M^*(K_{3,3}')$ then $e_i$ is edge $e'$.  
Since neither $F_7$ nor $F_7^*$ are frame, neither $N_1$ nor $N_2$ is isomorphic to $F_7$ or $F_7^*$.  
If $N_1 \twosume{e_1}{e_2} N_2 \in \Ee_0$, then by minimality, for $i \in \{1,2\}$, $M_i \iso N_i$ and $M_1 \twosume{e_1}{e_2} M_2 \iso N_1 \twosume{e_1}{e_2} N_2$.  
Otherwise, $N_1 \iso N_2 \iso U_{2,4}$, so both $M_1$ and $M_2$ are non-binary.  
\end{proof}

Our next lemma requires two basic facts.  
The first is a result of Bixby; the second was proved independently by Brylawski and Seymour.  

\begin{prop}[\cite{oxley:mt}, Proposition 11.3.7] 
\label{prop:Bixby_U24s_use_every_element} 
Let $M$ be a connected matroid having a $U_{2,4}$ minor and let $e \in E(M)$.  Then $M$ has a $U_{2,4}$ minor using $e$.  
\end{prop}
 
\begin{prop}[\cite{oxley:mt}, Proposition 4.3.6] 
\label{prop:Brylawski/Seymour}
Let $N$ be a connected minor of a connected matroid $M$ and suppose that $e \in E(M) \setminus E(N)$.  Then at least one of $M \setminus e$ and $M / e$ is connected and has $N$ as a minor.  
\end{prop} 

\begin{lem} \label{lem:oneisu24}
Let $M_1 \twosume{e_1}{e_2} M_2$ be an excluded minor for the class of frame matroids with both $M_1$ and $M_2$ non-binary.  
Then one of $M_1$ or $M_2$ is isomorphic to $U_{2,4}$.
\end{lem}

\begin{proof}
Suppose for a contradiction that neither $M_1$ nor $M_2$ is isomorphic to $U_{2,4}$.  
By Propositions \ref{prop:Bixby_U24s_use_every_element} and \ref{prop:Brylawski/Seymour} we may choose an element $f \in E(M_1) \setminus \{e_1\}$ so that a matroid $M_1'$ obtained from $M_1$ by either deleting or contracting $f$ is connected and has $U_{2,4}$ as a minor.  
Since $M_1' \twosume{e_1}{e_2} M_2$ is a minor of $M_1 \twosume{e_1}{e_2} M_2$, by minimality $M_1' \twosume{e_1}{e_2} M_2$ is frame.  
By Theorem \ref{thm:biased2sum}, $(M_2, \{e_2\})$ is frame.  
Similarly, $(M_1, \{e_1\})$ is frame.  
Hence by Thereom \ref{thm:biased2sum}, $M_1 \twosume{e_1}{e_2} M_2$ is frame, a contradiction.  
\end{proof}

The final lemma we need to prove Theorem \ref{thm:exclminr2sumu24} tells us that in our current setting, 2-separations having one side just a 3-circuit cannot interact with each other.  
The complement of a subset $A \subseteq E$ is denoted $A^c$.  

\begin{lem} \label{lem:disjtu24}
Let $M$ be a connected matroid on $E$ with $|E| \ge 6$ and assume that for every 2-separation $(A, A^c)$ of $M$, one of $M[A]$ or $M[A^c]$ is a circuit of size 3.  If $(A, A^c)$ and $(B, B^c)$ are 2-separations with both $M[A]$ and $M[B]$ a circuit of size 3, then either $A = B$ or $A \cap B = \emptyset$.
\end{lem}

\begin{proof}
Suppose for a contradiction that $\emptyset \neq A \cap B \neq A$.  
We consider two cases depending on the size of $A \cap B$.
If $|A \cap B| = 1$ then let $A \cap B = \{e\}$ and consider the separation $(A \setminus \{e\}, A^c \cup \{e\})$.  
Since $B \setminus \{e\}$ spans $e$, $r(A^c) = r(A^c \cup \{e\})$.  
But this implies that $(A \setminus \{e\}, A^c \cup \{e\})$ is a 2-separation, a contradiction as neither side has size three.

Next suppose $|A \cap B| = 2$.  
Then, summing the orders of the separations $(A \cap B, A^c \cup B^c)$ and $(A \cup B, A^c \cap B^c)$, by submodularity, we have 
\begin{align*} 
\lambda_M(A \cap B, &A^c \cup B^c)) + \lambda_M(A \cup B, A^c \cap B^c) \\ 
&= r(A \cap B) + r(A^c \cup B^c) + r(A \cup B) + r(A^c \cap B^c) - 2 r(M) + 2 \\
&\leq  r(A) + r(A^c) + r(B) + r(B^c) - 2 r(M) + 2 \\ 
&= \lambda_M(A,A^c) + \lambda_M(B,B^c)=4.
\end{align*}
As $M$ is connected, each of $\lambda_M(A \cap B, A^c \cup B^c))$ and $\lambda_M(A \cup B, A^c \cap B^c)$ is at least two, so this implies that $(A \cap B, A^c \cup B^c)$ is a 2-separation, again a contradiction.  
\end{proof}

\begin{proof} [Proof of Theorem \ref{thm:exclminr2sumu24}]
Let $M$ be an excluded minor for the class frame matroids, and suppose $M$ has connectivity 2 and $M \notin \Ee_0$.  
By Lemma \ref{lem:none0struc}, whenever $M$ is written as a 2-sum, each term of the sum is non-binary, and by Lemma \ref{lem:oneisu24} one of these terms is isomorphic to $U_{2,4}$.  
Hence every 2-separation $(A, A^c)$ of $M$ has one of $M[A]$ or $M[ A^c ]$ a circuit of size 3.  
By Lemma \ref{lem:disjtu24} the 3-circuits corresponding to these $U_{2,4}$ minors are pairwise disjoint.  
Therefore we may write $M = N \gltwosum{L}$, where $N$ is a 3-connected matroid. 
\end{proof}

\subsection{Excluded minors for the class of frame matroidals}
\label{sec:l_excluded_minors}

Theorem \ref{thm:exclminr2sumu24} says that every excluded minor of connectivity 2 for the class of frame matroids that is not in $\Ee_0$ can be expressed in the form $N \gltwosum{L}$, where $N$ is a 3-connected frame matroid.  
In this section we equate the problem of representing a matroid of this form as a biased graph to frame matroidals.  
We begin with the following key result.  

\begin{thm} \label{thm:u24eqlb}
Let $N$ be a matroid and let $L \subseteq E(N)$.  
Then $N \gltwosum{L}$ is frame if and only if the matroidal $(N,L)$ is frame.
\end{thm}

\begin{proof}
Let $L = \{e_1, \ldots, e_k\}$ and repeatedly apply Lemma \ref{lem:workhorse2sum}: 
\begin{align*}
N \gltwosum{L} \mbox{ is frame } 
	& \iff \Big( \big(N \gltwosum{ \{e_1 \ldots e_{k-1}\} } \big) 
		\twosume{e_k}{} U_{2,4} \, , \, \emptyset \Big) \mbox{ is frame }		\\
	& \iff \Big( \big(N \gltwosum{ \{e_1 \ldots e_{k-2} \} } \big) 
		\twosume{e_{k-1}}{} U_{2,4} \, , \, \{e_k\} \Big) \mbox{ is frame }		\\
	& \iff \Big( \big(N \gltwosum{ \{e_1 \ldots e_{k-3}\} } \big) 
		\twosume{e_{k-2}}{} U_{2,4} \, , \, \{e_{k-1},e_k\} \Big) \mbox{ is frame }		\\
	& \qquad\qquad\qquad\qquad \vdots	\\
	& \iff \big(N, \{e_1, \ldots, e_k\} \big)  \mbox{ is frame. }
\qedhere
\end{align*}
\end{proof}

So that we may work directly with matroidals, we now define minors of matroidals.  
Any matroidal $(N,K)$ obtained from a matroidal $(M,L)$ by a sequence of the operations of deleting or contracting an element not in $L$ or removing an element from $L$ is a \emph{minor} of $(M,L)$.  
Clearly, the class of frame matroidals is minor-closed, and so we may ask for its set of excluded minors.  
We have the following immediate corollary of Theorem \ref{thm:u24eqlb}.

\begin{cor} \label{cor:ellexmin}
Let $N$ be a matroid and let $L \subseteq E(N)$.  Then $N \gltwosum{L}$ is an excluded minor for the class of frame matroids if and only if $(N,L)$ is an excluded minor for the class of frame matroidals.  
\end{cor}

Our search for the remaining excluded minors of connectivity 2 for the class of frame matroids is therefore equivalent to the problem of finding excluded minors for the class of frame matroidals.  

There are four ways to represent the 3-circuit $U_{2,3}$ as a biased graph: a balanced triangle, a contrabalanced theta on two vertices, a tight handcuff consisting of an unbalanced 2-cycle together with an unbalanced loop, or as a loose handcuff consisting of a link and two unbalanced loops; no biased graph representation of $U_{2,3}$ has all three elements as unbalanced loops.  
Evidently therefore, $U_{2,3} \gltwosum{E(U_{2,3})}$ is not frame.  
Let us denote this matroid $N_9$. 
\Ie, \[ N_9 = U_{2,3} \gltwosum{ E(U_{2,3}) }.  \]

\begin{prop} \label{prop:C_3_all_edges_in_L_is_L_ex_minor} 
$N_9$ is an excluded minor for the class of frame matroids.
\end{prop}

\begin{proof} 
By Corollary \ref{cor:ellexmin}, 
$N_9$ is an excluded minor for the class of frame matroids \iiff $(U_{2,3}, E(U_{2,3}))$ is an excluded minor for the class of frame matroidals.  
There is no biased graph representing $U_{2,3}$ in which all three elements are unbalanced loops, so the matroidal $(U_{2,3},E(U_{2,3}))$ is not frame.  
For every two element subset $L \subseteq E(U_{2,3})$ the matroidal $(U_{2,3}, L)$ is frame: a link between two vertices together with an unbalanced loop on each endpoint, where the two unbalanced loops represent the two elements in $L$ is an $L$-biased graph representing $U_{2,3}$.   
\end{proof} 

The matroid $N_9$ is the only excluded minor for the class of frame matroids of the form $N \gltwosum{L}$ with $|L| \geq 3$:  

\begin{thm} \label{thm:Size_L_more_than_3}
Let $N$ be a 3-connected matroid, let $L \subseteq E(N)$, and suppose that $M = N \gltwosum{L}$ is an excluded minor for the class of frame matroids.  If $|L| \geq 3$ then $M \iso N_9$. 
\end{thm}

\begin{proof} 
Let $L = \{e_1, \ldots, e_k\}$.  By Corollary \ref{cor:ellexmin}, $(N,L)$ is an excluded minor for the class of frame matroidals.  
By minimality then, $(N, \{e_2, \ldots, e_k\})$ is frame.  
Let $\Om$ be a $\{e_2, \ldots, e_k\}$-biased graph representing $(N, \{e_2, \ldots, e_k\})$.   
In $\Omega$, edges $e_2, e_3$ are unbalanced loops and $e_1$ is a link.  
Since $N$ is 3-connected, $\Omega$ is 2-connected.  
Hence there are disjoint paths $P$, $Q$ linking the endpoints of $e_1$ and the vertices incident to $e_2$ and $e_3$.  
Contracting $P$ and $Q$ yields $U_{2,3}$ as a minor with $E(U_{2,3}) = \{e_1, e_2, e_3\}$.  
By minimality and Proposition \ref{prop:C_3_all_edges_in_L_is_L_ex_minor} therefore, $N \iso U_{2,3}$ and $L = \{e_1, e_2, e_3\}$.  
\end{proof}

\section{Proof of Theorem \ref{thm:two_sep_ex_min_main}} \label{sec:Main_theorem}

We are now ready to prove Theorem \ref{thm:two_sep_ex_min_main}.  

\bigskip 
\noindent \textbf{Theorem \ref{thm:two_sep_ex_min_main}.}  \hspace{-10pt}  
\textit{
Let $M$ be an excluded minor for the class of frame matroids, and suppose $M$ is not 3-connected.  
Then either $M$ is isomorphic to a matroid in $\Ee$ 
or $M$ is the 2-sum of a 3-connected non-binary frame matroid and $U_{2,4}$.  }
\bigskip 

The set $\Ee$ of excluded minors in the statement of Theorem \ref{thm:two_sep_ex_min_main} contains $\Ee_0$ and $N_9$.  
In this section we exhibit the remaining matroids in $\Ee$, and show that any other excluded minor of connectivity 2 is a 2-sum of a 3-connected non-binary matroid and $U_{2,4}$.  
We do this using matroidals.  
We show that the nine matroidals $\Mm_0, \ldots, \Mm_8$ illustrated in Figure \ref{matroidals1} are excluded minors for the class of frame matroidals.  
Each matroidal $\Mm_i {=} (M_i, L_i)$, $i \in \{0, \ldots, 8\}$, is given as the frame matroid $M_i = F(\Omega_i)$ represented by a biased graph $\Omega_i=(G,\Bb)$, where the graph $G$ is shown in Figure \ref{matroidals1} and collections $\Bb$ are as listed.  
Each matroidal's set $L_i$ is the set $\{e_1, e_2\}$, consisting of the pair of elements represented by edges $e_1$, $e_2$ in each graph.  

\begin{figure}[tbp] 
\begin{center} 
\includegraphics[scale=0.8]{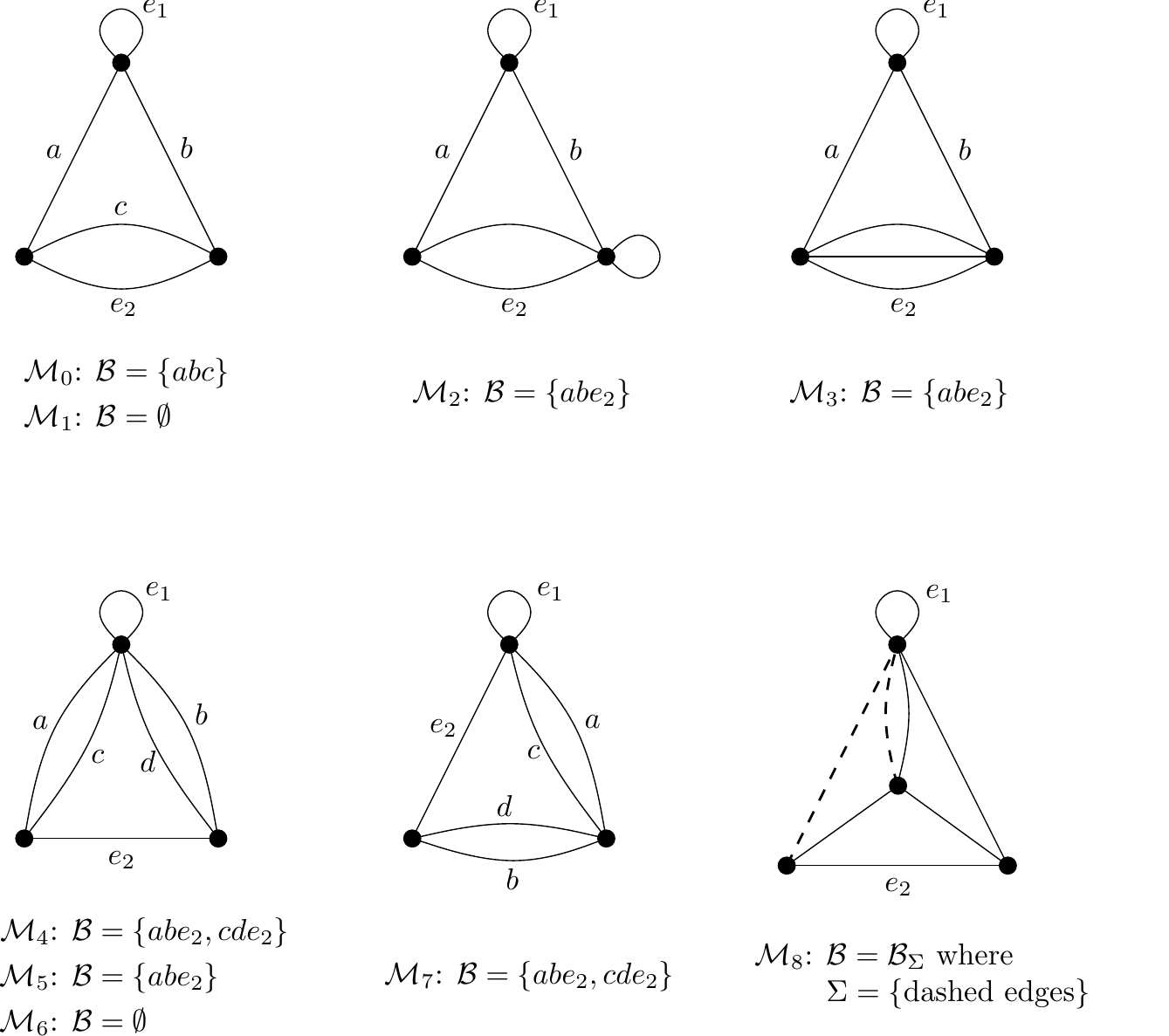}
\end{center} 
\caption{Excluded minors for the class of frame matroidals with $|L|>1$.}
\label{matroidals1}
\end{figure} 

Note that the excluded minor $N_9$ is given by the matroidal $\Mm_0{=}(M_0, L_0)$: $M_0 \gltwosum{L_0} \iso N_9$ (it is straightforward to verify that the circuits of $M_0 \gltwosum{L_0}$ and those of $N_9$ coincide).  
In fact, $U_{2,3}$ gives rise to four excluded minors for the class of frame matroidals, each yielding $N_9$ as corresponding excluded minor for the class of frame matroids, as follows.  
Write $E=E(U_{2,3})$, choose a subset $S \subseteq E$, and let $N = U_{2,3} \gltwosum{S}$.  
Set $L = E \setminus S$.  
Then $N \gltwosum{L} \iso N_9$ and matoridal $(N,L)$ is an excluded minor for the class of frame matroidals.  
The four choices for the size of $S$ each thereby yield an excluded minor for the class of frame matroidals.  

Matroidals $\Mm_1, \ldots, \Mm_8$ give rise to excluded minors for the class of frame matroids that we have not yet encountered.  
Let $\Ee_1 = \{ N \gltwosum{L} \mid (N,L) \in \{ \Mm_0, \ldots, \Mm_8\} \}$.  
Let $\Ee = \Ee_0 \cup \Ee_1$.  

The hard work in proving Theorem \ref{thm:two_sep_ex_min_main} is in showing that $\{\Mm_1, \ldots, \Mm_8\}$ is the complete list of excluded minors for the class of frame matroidals having $|L|=2$.  
This is the content of Lemma \ref{lem:workhorse}.  

\begin{lem} \label{lem:workhorse}
Let $N$ be a 3-connected matroid and let $L \subseteq E(N)$ with $|L| = 2$.  If $(N,L)$ is an excluded minor for the class of frame matroidals, then it is isomorphic to one of $\Mm_1, \ldots, \Mm_8$.
\end{lem}

Before proving Lemma \ref{lem:workhorse}, let us show that it implies Theorem \ref{thm:two_sep_ex_min_main}.  

\begin{proof}[Proof of Theorem \ref{thm:two_sep_ex_min_main}]
Let $M$ be an excluded minor for the class of frame matroids, and suppose $M$ is not 3-connected.  
By Theorem \ref{thm:exclminr2sumu24}, either $M$ is isomorphic to a matroid in $\Ee_0$ or $M = N \gltwosum{L}$ for a 3-connected frame matroid $N$ and a nonempty set $L$.  
So suppose $M \notin \Ee_0$.  
By Theorem \ref{thm:Size_L_more_than_3}, if $|L| \geq 3$ then $M \iso N_9$.  
If $|L|=1$ then $M$ is a 2-sum of $N$ and $U_{2,4}$, and by Lemma \ref{lem:none0struc} $N$ is a non-binary.  
Finally, if $|L|=2$, then by Corollary \ref{cor:ellexmin} $(N,L)$ is an excluded minor for the class of frame matroidals.  
By Lemma \ref{lem:workhorse}, $M$ is isomorphic to a matroid in $\Ee_1$.  
\end{proof}

\subsection{The excluded minors $\Ee_1$} 

Let us substantiate our claim that the matroids in $\Ee_1$ are excluded minors for the class of frame matroids.  

Say a matroid $M$ \emph{series reduces} to a matroid $M'$ if $M'$ may be obtained from $M$ by repeatedly contracting elements contained in a nontrivial series class.  
Series reduction of matroids is useful because matroidals consisting of a rank 2 matroid with a distinguished subset $L$ of size 2 are aways frame: 

\begin{lem} \label{lem:rank2L2matroidalsareframe}
Let $(N,L)$ be a matroidal.  If $N$ has rank $2$ and $|L| = 2$, then $(N,L)$ is frame.
\end{lem}

\begin{proof}
We may assume $N$ has no loops.  
Let $L=\{e_1, e_2\}$.  
Since $N$ has rank 2, $N$ is obtained from some uniform matroid $U_{2,m}$ by adding elements in parallel.  
We may assume that either $e_1, e_2 \in E(U_{2,m})$ or that $e \in E(U_{2,m})$ and $e_1$ and $e_2$ are in the same parallel class.  
Let $\Om$ be the contrabalanced biased graph representing $U_{2,m}$ with $V(\Om)=\{u,v\}$, $e_1$ a loop incident to $u$, $e_2$ a loop incident to $v$ if $e_2 \in E(U_{2,3})$, and all other elements represented by $u$-$v$ edges.  
Let $\Om'$ be the biased graph obtained by adding each element $f \in E(N) \setminus E(U_{2,m})$ in the same parallel class as an element $e \not= e_1$ as a $u$-$v$ edge and declaring circuit $ef$ balanced, and adding each element in $E(N) \setminus E(U_{2,3})$ in the same parallel class as $e_1$ as an unbalanced loop incident to $u$.  
Then $\Om'$ is an $L$-biased representation of $N$.  
\end{proof}

This tool in hand, we may now prove:

\begin{prop} \label{prop:excluded_matroidals} 
The matroidals $\Mm_0, \ldots, \Mm_7$ are excluded minors for the class of frame matroidals.
\end{prop}

\begin{proof} 
That $\Mm_0$ is an excluded minor follows immediately from Corollary \ref{cor:ellexmin}, Proposition \ref{prop:C_3_all_edges_in_L_is_L_ex_minor}, and the fact that $M_0 \gltwosum{L_0} \iso N_9$.  
So suppose for a contradiction that for some $i \in \{1, \ldots, 8\}$, $\Om$ is a $L_i$-biased graph representing $\Mm_i {=} (M_i, L_i) \in \{\Mm_1, \ldots, \Mm_7\}$.  
Let $e_1, e_2$ be the elements in $L_i$.   
For $j \in \{1,2\}$, let $v_j$ be the vertex of $\Om$ incident to $e_j$.  
Since $\{e_1, e_2\}$ is not a circuit, $v_1 \not= v_2$.  
Since each of $M_1, \ldots, M_7$ has rank 3, $|V(\Om)|=3$; let $u$ be the third vertex of $\Om$.  
Since none of $M_1, \ldots, M_7$ has a circuit of size three containing $e_1$ and $e_2$, there cannot be an edge linking $v_1$ and $v_2$.  
But then $u$ is a cut-vertex of $\Omega$, a contradiction since all of $M_1, \ldots, M_7$ are 3-connected.  

We now show that every proper minor of each of $\Mm_1, \ldots, \Mm_7$ is frame.  
The biased graphs shown in Figure \ref{matroidals1} show that each matroidal $(M_i, L_i \setminus e_2)$ is frame ($i \in \{1, \ldots, 7\}$).  
The biased graphs shown in Figure \ref{matroidals1_alt_reps} 
\begin{figure}[tbp] 
\begin{center} 
\includegraphics[scale=0.8]{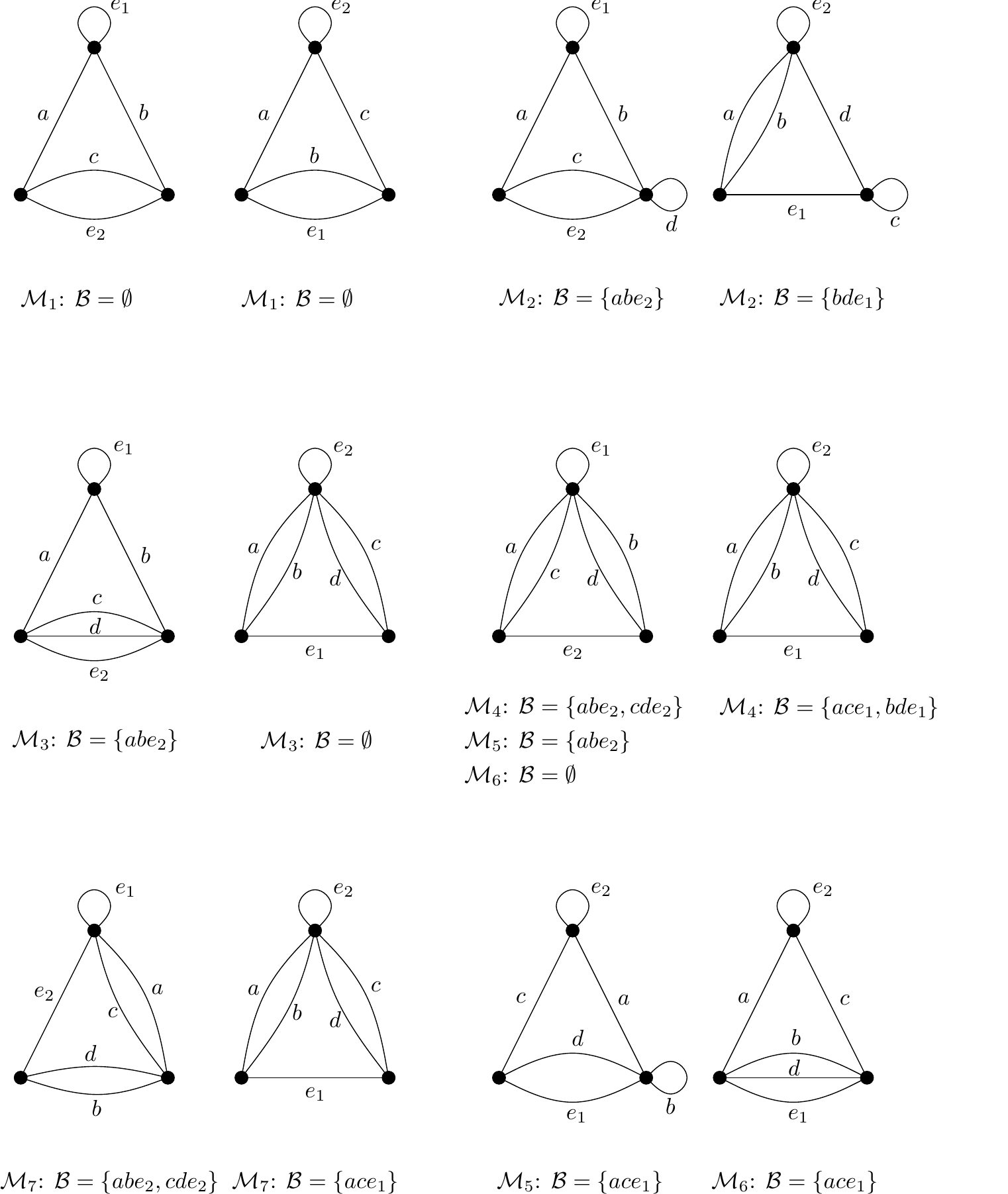}
\end{center} 
\caption{Alternate representations of $\Mm_1, \ldots, \Mm_7$.}
\label{matroidals1_alt_reps}
\end{figure} 
show that also each matroidal $(M_i, L_i \setminus e_1)$ is frame.  
Any matroidal $(N, L)$ obtained from one of $\Mm_1, \ldots, \Mm_7$ by contracting an element other than $e_1$ or $e_2$ has matroid $N$ of rank 2, and so is frame by Lemma \ref{lem:rank2L2matroidalsareframe}.  
Finally, suppose that $(N,L)$ is a matroidal obtained from one of $\Mm_1, \ldots, \Mm_7$ by deleting an element $e$ other than $e_1,e_2$.  
In all cases, the resulting matroid series reduces to a matroid $N'$ of rank 2 with both $e_1, e_2 \in E(N')$ by the contraction of a single edge $s$ (this is easy to see by considering the biased graph representations of Figure \ref{matroidals1_alt_reps}: in each case, one of the biased graphs representing $M_i \setminus e$ obtained by deleting an edge $e \in \{a, b, c, d\}$ has a vertex incident to just two edges).  
By Lemma \ref{lem:rank2L2matroidalsareframe} therefore, there is an $L$-biased representation $\Om'$ of the series reduced matroid $N'$.  
Now let $\Om$ be a biased graph obtained from $\Om'$ by placing an edge representing $s$ in series with the other edge $t$ in its series class in $M_i \setminus e$ | that is, if $t$ is a link, subdivide $t$ to produce a path consisting of edges $s$ and $t$, and if $t$ is a loop, say incident to $v$, add a vertex $w$, add $s$ as a $v$-$w$ link, and place $t$ as a loop incident to $w$.  
Evidently this corresponds to a coextention of $N'$ to recover $N$, and $\Om$ is an $L$-biased representation of $N$.  
\end{proof}

\begin{prop} \label{lem:Ex_min_MK5_minus_matching_L_size_2}
The matroidal $\Mm_8$ is an excluded minor for the class of frame matroidals.
\end{prop} 

\begin{proof} 
The matroid of $\Mm_8$ is the rank 4 wheel, \ie\ the cycle matroid $M(W_4)$ where $W_4$ is the five vertex simple graph having one vertex of degree 4 and four vertices of degree 3 (Figure \ref{fig:W4}).  
\begin{figure}[tbp] 
\begin{center} 
\includegraphics[scale=0.8]{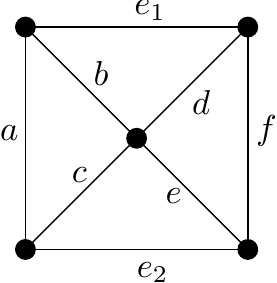}
\end{center} 
\caption{$W_4$.}
\label{fig:W4} 
\end{figure} 
Its distinguished subset $L = \{e_1, e_2\}$ consists of two nonadjacent edges both of which have both ends of degree three.  
(Pinch the two ends of $e_1$ to obtain the biased graph representation shown in Figure \ref{matroidals1}.)
We first show that $\Mm_8$ is not frame.  
Suppose for a contradiction that $M(W_4) \iso F(\Om)$ for some $L$-biased graph $\Om$.  
Then $e_1$ and $e_2$ are both unbalanced loops in $\Om$; say $e_i$ is incident to vertex $u_i$ ($i \in \{1, 2\}$).  
There is a unique circuit $C$ of size 4 in $M(W_4)$ containing $\{e_1, e_2\}$; say $C=e_1 e_2 f f'$.  
Elements $f, f'$ must form a path of length 2 in $\Om$ linking $u_1$ and $u_2$, say with interior vertex $v$.  
Since $\Om$ is not balanced, $|V(\Om)|=4$; let $v'$ be the fourth vertex of $\Om$.  
Note that since $e_1, e_2$ are not in a circuit of size 3 and not in any other circuit of size 4, all four remaining edges (other than $e_1, e_2, f, f'$) must be incident to $v'$.  
Since $M(W_4)$ has no elements in series or in parallel, there must be an edge with ends $u_1,v'$ and another edge with ends $u_2,v'$.  
This yields another 4-circuit in $F(\Om)$ containing $e_1$ and $e_2$, a contradiction.

We now show that every proper minor of $\Mm_8$ is frame.  
For $i \in \{1,2\}$, an $(L \setminus e_i)$-biased graph is obtained by pinching the ends of $e_{3-i}$ in the graph $W_4$, so the matroidal $(M(W_4), L \setminus e_i)$ is frame.  
Now consider a matroidal obtained from $\Mm_8$ by deleting or contracting an element not in $L$.  
Up to symmetry there are only two such edges to consider, say elements $d$ and $f$ as shown in Figure \ref{fig:W4}.  
The biased graphs of Figure \ref{fig:Mm8isexminor} show that deleting or contracting either of $d$ or $f$ yields a frame matroidal.  
These $L$-biased graphs may be obtained as follows.  
\begin{figure}[tbp] 
\begin{center} 
\includegraphics[scale=0.8]{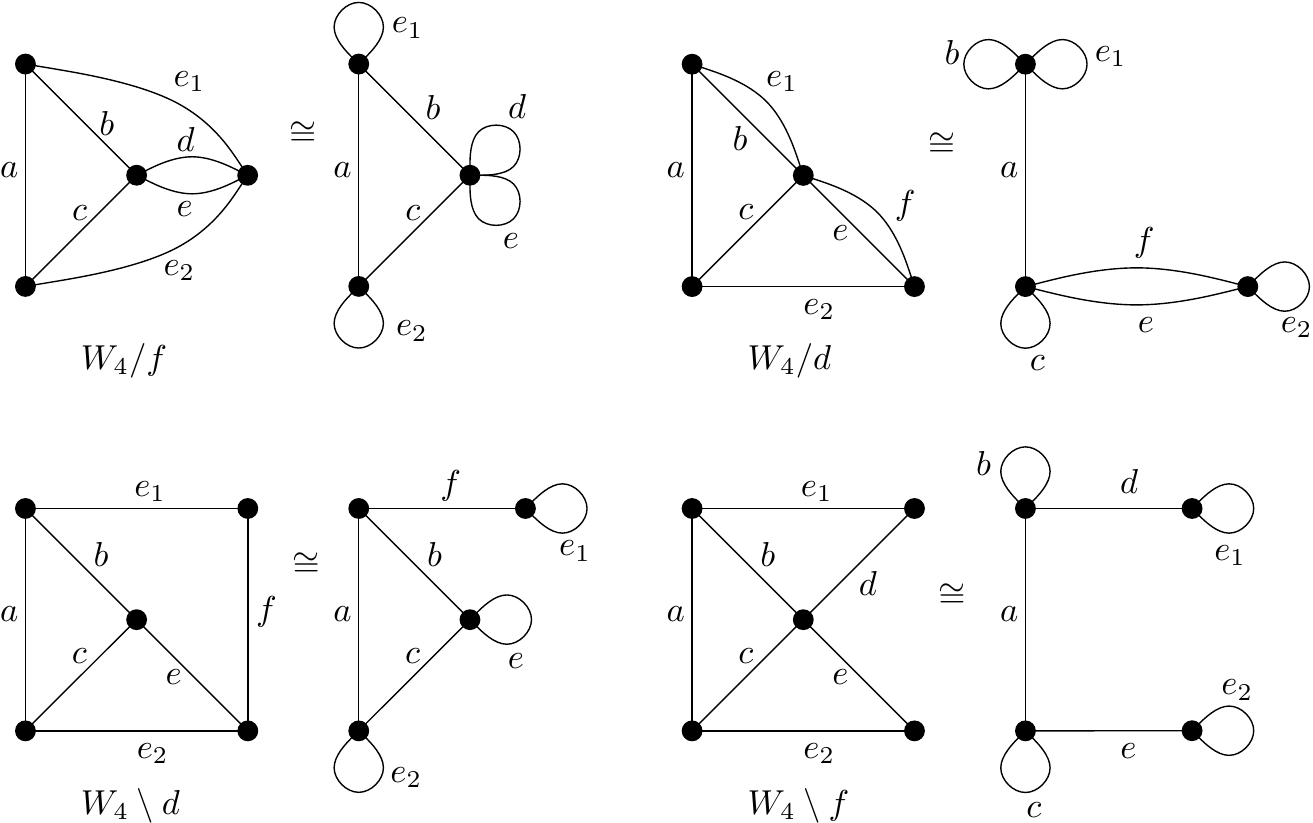}
\end{center} 
\caption{Any proper minor of $W_4$ is $\{e_1,e_2\}$-biased.}
\label{fig:Mm8isexminor} 
\end{figure} 
\begin{itemize} 
\item  Contracting $f$ in $W_4$ yields a graph in which $e_1$ and $e_2$ are incident to a common vertex.  Rolling up the edges incident to that vertex yields an $\{e_1, e_2\}$-biased graph, so $(M(W_4)/f, \{e_1, e_2\})$ is frame.  
\item  In $M(W_4)/d$ elements $\{e,f\}$ are parallel.  In $M(W_4)/d \setminus f$, elements $e$ and $e_2$ are in series, so $M(W_4)/d$ is represented by the graph obtained from $W_4 /d$ by replacing $e_2$ with the pair of parallel edges $e, f$ and replacing the pair $e, f$ with $e_2$.  
This yields a graph in which $e_1$ and $e_2$ are incident to a common vertex $v$.  Now rolling up the edges in $\delta(v)$ yields an $\{e_1, e_2\}$-biased graph representing $M(W_4)/d$, so $(M(W_5)/d, \{e_1, e_2\})$ is frame.  
\item  In $M(W_4) \setminus d$ elements $e_1$ and $f$ are in series, so the biased graph $\Om$ obtained from $W_4 \setminus d$ by swapping edges $e_1$ and $f$ represents $M(W_4) \setminus d$.  
Since $e_1$ and $e_2$ are incident to a common vertex $v$ in $\Om$, the biased graph obtained by rolling up the edges in $\delta(v)$ is an $\{e_1,e_2\}$-biased graph representing $M(W_4) \setminus d$.  
\item  Similarly, $M(W_4) \setminus f$ has series classes $\{e_1,d\}$ and $\{e_2,e\}$.  Hence swapping edges $e_1$ and $d$, and swapping edges $e_2$ and $e$, we obtain a biased graph representing $M(W_4) \setminus f$ in which $e_1$ and $e_2$ are incident to a common vertex.  Rolling up the edges $e_1, e_2, b, c$ incident to that vertex yields an $\{e_1, e_2\}$-biased graph representing $M(W_4) \setminus f$.  
\qedhere
\end{itemize} 
\end{proof}

\subsection{Finding matroidal minors using configurations} 

To prove Lemma \ref{lem:workhorse}, we suppose $\Nn{=}(N,L)$ is an excluded minor for the class of frame matroidals with $|L|=2$ that is not one of $\Mm_1, \ldots, \Mm_8$.  
We then work with a biased graph $\Psi$ representing $N$ to derive the contradiction that $(N,L)$ contains one of $\Mm_0, \ldots, \Mm_8$ as a minor.  
When doing so, we are looking for biased graphs representing one of $\Mm_0, \ldots, \Mm_8$.  
Some of $\Mm_0, \ldots, \Mm_8$ share the same underlying graphs or have an underlying graph contained in the underlying graph of another (Figures \ref{matroidals1} and \ref{matroidals1_alt_reps}).  
Since which of $\Mm_0, \ldots, \Mm_8$ we find as a minor of $\Nn$ is irrelevant, it is enough to determine the underlying graph of a minor of $\Psi$ along with just enough information about the biases of its cycles to see that $\Psi$ must contain one of $\Mm_0, \ldots, \Mm_8$ as a minor.  
We formalize this as follows.  

A \emph{configuration} $\Cc$ consists of a graph $G$ with two distinguished edges $e_1, e_2$, together with a set $\Uu$ of cycles of $G$, which we call \emph{unbalanced}.  
The configurations we find are those named $\Cc_1$, \ldots, $\Cc_4$, $\Cc_4'$, $\Cc_4''$, $\Cc_5$, \ldots, $\Cc_8$ in Figure \ref{fig:configs_c}, and $\Dd_1$, $\Dd_2$, $\Dd_2'$, and $\Dd_3$ in Figure \ref{fig:configs_d}.  
\begin{figure}[tbp] 
\begin{center} 
\includegraphics[scale=0.8]{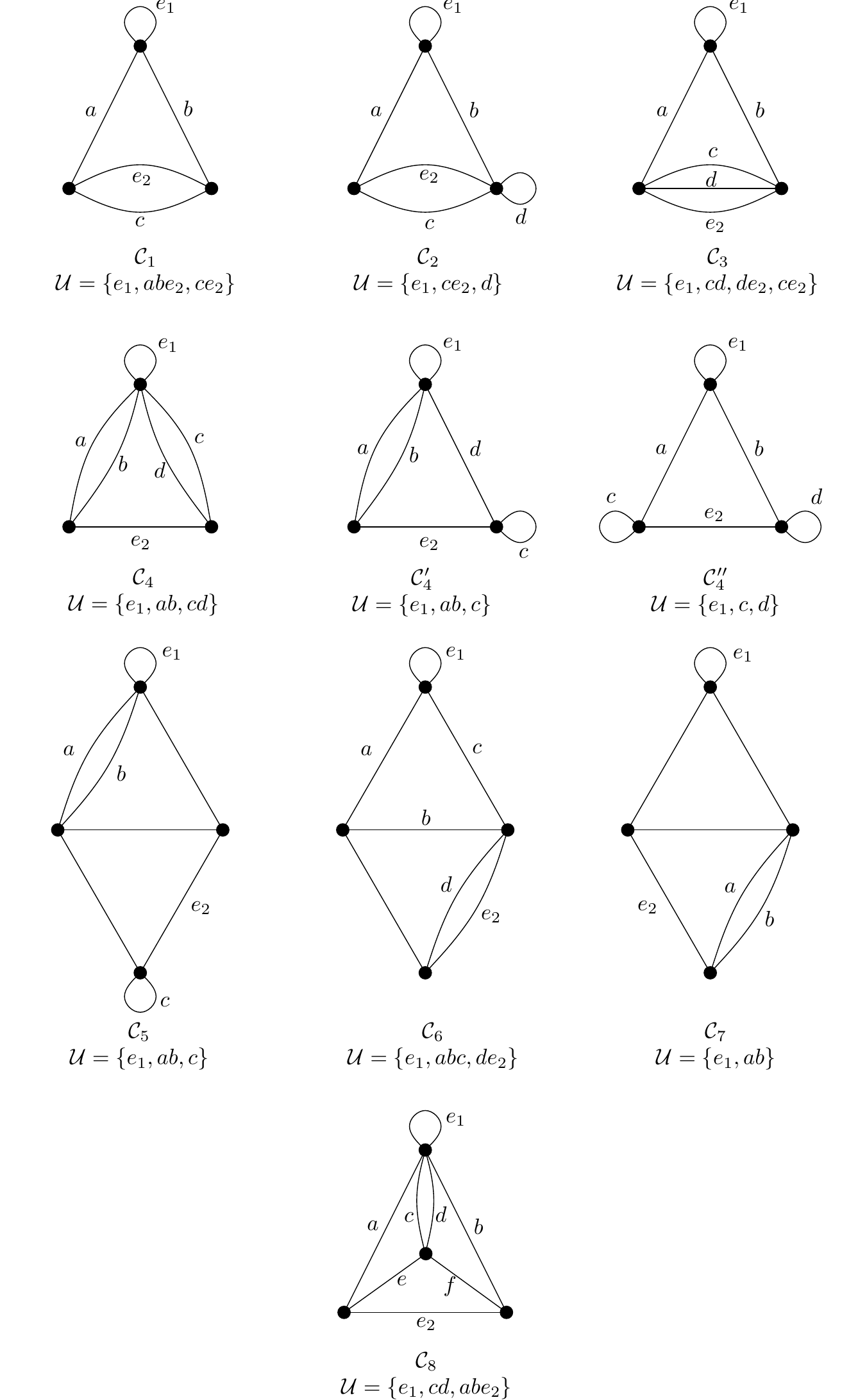}
\end{center} 
\caption{Configurations used to find $\Mm_0, \ldots, \Mm_8$.}
\label{fig:configs_c}
\end{figure} 
We say that a biased graph $\Omega {=} (G, \Bb)$ \emph{realises} configuration $\Cc{=}(G,\Uu)$ if $\Bb \cap \Uu = \emptyset$.  
The following two lemmas guarantee that finding one of these configurations in $\Psi$ implies that $\Nn$ contains one of $\Mm_0, \ldots, \Mm_8$ as a minor.  

\begin{lem} \label{lem:cconfig}
Let $\Omega$ be a biased graph that realises one of the configurations $\Cc_1,\ldots,\Cc_4$, $\Cc_4'$, $\Cc_4''$, $\Cc_5,\ldots,\Cc_8$.  
Then $(F(\Omega), \{e_1,e_2\})$ contains one of $\Mm_0, \ldots, \Mm_8$ as a minor.
\end{lem}

\begin{proof}
We show that in each case, $\Om$ has a minor containing $\{e_1, e_2\}$ isomorphic to one of the biased graphs $\Om_i$ representing the matroid $M_i$ of a matroidal $\Mm_i$ ($i \in \{0, \ldots, 8\}$).  
This implies that $F(\Om)$ has $M_i$ as a minor containing $\{e_1, e_2\}$, and so that $(F(\Om), \{e_1,e_2\})$ contains $\Mm_i$ as a minor.  
Recall that the biased graphs $\Om_i$ defining $M_i$ ($i \in \{0, \ldots, 8\}$) are those shown in Figure \ref{matroidals1}.  

The only two realisations of $\Cc_1$ are the biased graphs $\Om_0$ and $\Om_1$ representing the matroids $M_0$ of $\Mm_0$ and $M_1$ of $\Mm_1$.  
A biased graph realising $\Cc_2$  (resp.\ $\Cc_3$) will have a subgraph realising $\Cc_1$ unless it is isomorphic to $\Om_2$ (resp.\ $\Om_3$).  
A biased graph realising $\Cc_4$ has either 0, 1, or 2 balanced cycles, and so is isomorphic to one of $\Om_4$, $\Om_5$, or $\Om_6$, respectively.  
If $\Omega$ is a biased graph realising $\Cc_4'$ or $\Cc_4''$ then $\Om$ has a unique balancing vertex after deleting its unbalanced loops; unrolling its unbalanced loops we obtain a biased graph $\Phi$ realising $\Cc_4$ with $F(\Phi) \iso F(\Om)$.  

Suppose $\Omega$ realises $\Cc_5$.  
Let $a, b$ be the two parallel edges forming the unbalanced cycle.  
We may assume by possibly interchanging $a$ and $b$ that the unique triangle containing $a$ is unbalanced.  
Contracting $a$ and deleting $b$ yields a $\Cc_4'$ configuration.  

Suppose $\Omega$ realises $\Cc_6$.  
Then by the theta property there is an unbalanced cycle either of length 3 or length 4 containing $e_2$.  
In either case, this unbalanced cycle together with unbalanced cycle $d e_2$ has a minor that is a $\Cc_1$ configuration.  

If $\Omega$ realises $\Cc_7$, then | since by the theta property one of $a$ or $b$ is in an unbalanced triangle | contracting one of edges $a$ or $b$ we obtain a $\Cc_2$ configuration.

Finally suppose that $\Omega$ realises $\Cc_8$.  
If the triangle $e f e_2$ is unbalanced, then deleting $c, d$ and contracting one of the edges now in series yields configuration $\Cc_1$.  
So suppose triangle $efe_2$ is balanced.  
If one of $c$ or $d$ | say $d$ | fails to be contained in a balanced triangle, then deleting $c$ and contracting $d$ yields configuration $\Cc_4$.  
The remaining possibility is that $e f e_2$ is balanced and both $c$ and $d$ are contained in a balanced triangle.  
Then $\Omega$ may be embedded in the plane as drawn in Figure \ref{fig:configs_c} with precisely facial cycles $e f e_2$, $ace$, and $bdf$ balanced.  
The theta property implies that every cycle of length $>1$ in this graph is unbalanced if in the embedding its interior contains the face bounded by unbalanced cycle $cd$, and is otherwise balanced.  
Hence $\Om \iso \Om_8$.  
\end{proof}

\begin{figure}[tbp] 
\begin{center} 
\includegraphics[scale=0.8]{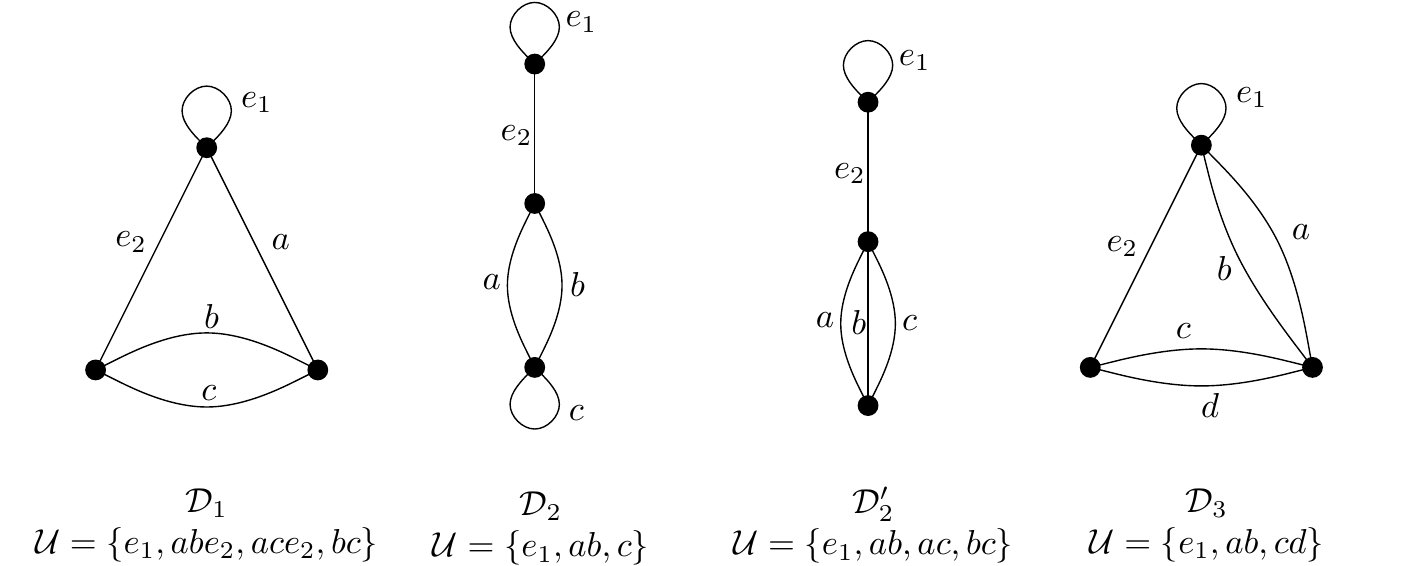}
\end{center} 
\caption{More configurations.}
\label{fig:configs_d}
\end{figure} 

\begin{lem} \label{lem:dconfig}
Let $\Omega$ be a biased graph which realises one of the configurations $\Dd_1$, $\Dd_2$, $\Dd_2'$, or $\Dd_3$.
Then $(F(\Omega), \{e_1,e_2\})$ contains one of $\Mm_0$, $\Mm_1$, or $\Mm_7$ as a minor.
\end{lem}

\begin{proof}
If $\Omega$ realises $\Dd_1$ then $F(\Omega) \iso M_1$, so $(F(\Omega), \{e_1,e_2\})$ is isomorphic to $\Mm_1$.  
If $\Omega$ realises either $\Dd_2$ or $\Dd_2'$ then $F(\Om) \iso M_0$, so $(F(\Omega), \{e_1,e_2\})$ is isomorphic to $\Mm_0$.  
If $\Omega$ realises $\Dd_3$, then either $\Omega$ contains a $\Dd_1$ configuration or $F(\Om) \iso M_7$ so $(F(\Om), \{e_1, e_2\})$ is isomorphic to $\Mm_7$.
\end{proof}

Two of the excluded minors for the class of frame matroidals have graphic matroids, namely $\Mm_4$ and $\Mm_8$: 
$M_4$ is the cycle matroid of $K_4$ and $M_8$ is the rank 4 wheel.  
The following lemma will help us locate either $\Mm_4$ or $\Mm_8$ as a minor in a purported excluded minor $\Nn{=}(N,L)$ in which $N$ is graphic.  

\begin{lem} \label{lem:On_rooted_K4_minors} 
Let $G$ be a simple 3-connected graph, and let $\{e_1, e_2\} \subseteq E(G)$, with $e_1 {=} s_1 t_1$ and $e_2 {=} s_2 t_2$, with $s_1, s_2, t_1, t_2$ pairwise distinct.  
Then either $G$ has a $K_4$ minor containing $\{e_1, e_2\}$ in which $e_1$ and $e_2$ do not share an endpoint, or $G$ has $W_4$ as a minor containing $\{e_1, e_2\}$ in which $e_1$ and $e_2$ are opposite each other in the rim of $W_4$ (\ie, $e_1$ and $e_2$ do not share an endpoint and each of $e_1$ and $e_2$ have both endpoints of degree three).  
\end{lem} 

\begin{proof} 
Let $\co(H)$ denote the graph obtained from a graph $H$ by suppressing vertices of degree 2.  
It is well known that if $G$ is a 3-connected graph, then for every $e \in E(G)$, either $\co(G \setminus e)$ or $G/e$ is 3-connected (for instance, it is a special case of Proposition 8.4.6 in \cite{oxley:mt}).  
In the following, if in $G \setminus e$ edge $e_i$, $i \in \{1,2\}$, has an endpoint of degree two, then $\co(G \setminus e)$ is obtained by contracting the edge other than $e_i$ incident to that vertex.  

Let $G$ be a minimal counter-example to the statement of the lemma.  
If there is an edge $e \in E(G)$ such that $\co(G \setminus e)$ or $G /e$ is 3-connected such that $e_1$ and $e_2$ are not incident to a common vertex, then by minimality this graph has a minor of one of the required forms.  
But then so would $G$ have had that minor, a contradiction.  
Hence for every edge $e \notin \{e_1, e_2\}$, if $\co(G \setminus e)$ is 3-connected then $e_1$ and $e_2$ are adjacent in $\co(G \setminus e)$, and if $G/e$ is 3-connected then $e_1$ and $e_2$ are adjacent in $G/e$.  

Suppose there is an edge $e \in E(G)$ that does not have any of $s_1, t_1, s_2, t_2$ as an endpoint.  
Then $\co(G \setminus e)$ has $e_1$ and $e_2$ nonadjacent, and so is not 3-connected.  
Hence $G/e$ is 3-connected.  
But neither are $e_1$ and $e_2$ adjacent in $G/e$, contradicting the previous paragraph.  
Therefore every edge of $G$ has an endpoint incident to $e_1$ or $e_2$.  
Now suppose $e \in E(G)$ does not have both endpoints in $\{s_1, t_1, s_2, t_2\}$; say $e = x s_1$ with $x \notin \{s_1, t_1, s_2, t_2\}$.  
Then $G/e$ does not have $e_1$ and $e_2$ adjacent, and so is not 3-connected.  
Hence $\co(G \setminus e)$ is 3-connected, and so has $e_1$ and $e_2$ adjacent.  
This implies that the degree of $s_1$ is three, and the three edges incident to $s_1$ are $e$, $e_1$, and $f$, where the other endpoint of $f$ is one of $s_2$ or $t_2$.  
It follows that $|V(G)| \leq 5$.  (Every vertex $x \notin \{s_1, t_1, s_2, t_2\}$ has neighbourhood of size $\geq 3$ contained in $\{s_1, t_1, s_2, t_2\}$.  Further, each vertex in the neighbourhood of $x$ has degree three, which, together with its edge to $x$ and its incident edge in $\{e_1, e_2\}$, includes an edge whose other endpoint is also in $\{s_1, t_1, s_2, t_2\}$.  These edges resulting from the existence of $x \notin \{s_1, t_1, s_2, t_2\}$ accounted for thus far leave just one vertex $u$ in $\{s_1, t_1, s_2, t_2\}$ for which it is possible that $u$ has an additional incident edge, yet the existence of a vertex $y \notin \{x, s_1, t_1, s_2, t_2\}$ requires three such vertices.)  

If $|V(G)|=4$, then $G \iso K_4$ and we are done.  
So suppose $|V(G)|=5$; let $V(G) = \{x, s_1, t_1, s_2, t_2\}$.  
The fact that the degree of every vertex is at least three, together with the above constraints on edges incident to a neighbour of $x$ forces the existence of either a $K_4$ or $W_4$ minor of the required form.  
This contradiction completes the proof.  
\end{proof}

\subsection{Proof of Lemma \ref{lem:workhorse}}

If a biased graph $\Om$ has a minor realising a configuration, we say $\Om$ \emph{contains} the configuration.  
Let us call the configurations $\Cc_1$, \ldots, $\Cc_4$, $\Cc_4'$, $\Cc_4''$, $\Cc_5$, \ldots, $\Cc_8$, $\Dd_1$, $\Dd_2$, $\Dd_2'$, $\Dd_3$ \emph{bad} configurations.  
Thus by Lemmas \ref{lem:cconfig} and \ref{lem:dconfig}, if $\Om$ represents $M$, and $\Om$ contains a bad configuration, then the matroidal $(M,\{e_1,e_2\})$ has one of $\Mm_0, \ldots, \Mm_8$ as a minor.  

\begin{proof}[Proof of Lemma \ref{lem:workhorse}]
Let $\Nn{=}(N,L)$ be an excluded minor for the class of frame matroidals with $N$ 3-connected and $L {=} \{e_1,e_2\}$, and suppose $\Nn$ is not isomorphic to one of $\Mm_1, \ldots, \Mm_8$.  
Observe that $\Nn$ cannot have $\Mm_0$ as a minor, since then minimality would imply that $\Nn \iso \Mm_0$; since $e_1$ and $e_2$ are in series in $M_0$ this would contradict the fact that $N$ is 3-connected.  
In light of this and Lemmas \ref{lem:cconfig} and \ref{lem:dconfig} it suffices to derive the contradiction that a biased graph $\Om$ representing $N$ contains a bad configuration.  

First suppose that $N$ is graphic.  
Let $H$ be a graph with $N = M(H)$.  As $N$ is 3-connected, $H$ is simple and 3-connected.  
Hence neither $e_1$ nor $e_2$ is a loop in $H$.  
If edges $e_1$ and $e_2$ share an endpoint $v \in V(H)$, then rolling up the edges incident to $v$ yields a biased graph in which both $e_1$ and $e_2$ are unbalanced loops, a contradiction.  
Hence $e_1$ and $e_2$ do not share an endpoint.  
By Lemma \ref{lem:On_rooted_K4_minors} therefore, $H$ has a minor $H'$ isomorphic to either $K_4$ with $e_1$ and $e_2$ nonadjacent, or isomorphic to $W_4$ with $e_1$ and $e_2$ nonadjacent and neither incident to the vertex of degree 4.  
In the former case $\Nn$ contains $\Mm_4$ as a minor, and in the latter $\Mm_8$ as a minor, both contradictions.  

So $N$ is not graphic.  
Let $\Om{=}\GB$ be a biased graph representing $(N,\{e_1\})$.  
Since $N$ is 3-connected: 
\begin{itemize}
\item[(C1)]  $\Omega$ is 2-connected, and 
\item[(C2)]  if $(A,B)$ is a separation of $N$ with $|A| \ge 2$ and $\Omega[A]$ is balanced, then $|V(A) \cap V(B)| \geq 3$. 
\end{itemize}

Let $v$ be the vertex to which $e_1$ is incident.  
We consider two cases, depending on whether $e_1$ and $e_2$ are adjacent in $\Omega$.

\subsubsection{Case 1. $e_1$ and $e_2$ are not adjacent} 

Let $u, w$ be the endpoints of $e_2$.  
We consider three subcases depending on the behaviour of unbalanced cycles in $\Omega - v$.

\paragraph{Subcase (i) \ $\Omega - v$ has no unbalanced cycle of length $> 1$}  \mbox{} 

If $\Om-v$ contains unbalanced loops, then unrolling them yields an $\{e_1\}$-biased graph representing $N$ in which $v$ is a balancing vertex.  
We may assume therefore that $\Omega - v$ is balanced.  
Consider the balancing equivalence classes in $\delta(v)$.
There cannot be just one b-class in $\delta(v)$, since then $e_1$ would not be contained in any circuit of $N$.  
If there are only two b-classes, then by Proposition \ref{prop:If_no_odd_theta} $\Omega$ is a signed graph.  
But then splitting $v$ yields a graph $H$ with $M(H) = N$ (Proposition \ref{prop:vertex_splitting_operation}), so $N$ is graphic, a contradiction.  
Hence there are at least three b-classes in $\delta(v)$.  
\begin{claim} 
$\Om$ contains a $\Cc_4$ configuration.  
\end{claim}

\begin{proof}[Proof of claim] 
Construct an auxiliary graph $G$ from the underlying graph of $\Omega - e_1$, as follows.  
Let $\{ S_1, \ldots, S_t \}$ be the partition of $\delta(v)$ into its b-classes.  
Add a set of new vertices $X = \{ x_1, \ldots, x_t \}$, and, for each $i \in \{1, \ldots, t\}$, redefine the endpoints of each edge $f{=}xv \in S_i$ so that $f$ has endpoints $x, x_i$.  
Add a new vertex $y$ to $G$ that is adjacent to every vertex which is a neighbour of either $u$ or $w$.  

We claim that $G$ contains three vertex disjoint paths between $X$ and $\{u,w,y\}$.  
For if not, then by Menger's Theorem there exists a pair of subgraphs $G_1,G_2 \subseteq G$ whose edges partition $E(G)$ so that  $X \subseteq V(G_1)$ and $\{u,w,y\} \subseteq V(G_2)$ and $V(G_1) \cap V(G_2) = Z$ with $|Z| = 2$.  
If $Z$ contains at most one vertex of $X$ then the subgraph of $\Omega$ induced by $E(G_2)$ is balanced, contradicting (C2).  
Hence $Z$ contains two vertices of $X$.  
But this implies $v$ is a cut vertex of $\Omega$, contradicting (C1).  
This establishes the existence of our paths.  

So we may now assume that in $\Om$ there exist three internally disjoint paths, $P_1$ and $P_2$ from $v$ to $u$ and $P_3$ from $v$ to $w$ such that the three edges of these paths in $\delta(v)$ are in distinct b-classes.  
If there exists a path $Q$ from $P_1 \cup P_2$ to $P_3$ which is disjoint from $\{u,v\}$, then a minor of $P_1 \cup P_2 \cup P_3 \cup Q \cup \{e_1,e_2\}$ contains a $\Cc_4$ configuration.   

If there is no such path $Q$, then there is a partition $(A,B)$ of $E(\Omega)$ with $V(A) \cap V(B) = \{u,v\}$, $P_1, P_2 \subseteq \Omega[A]$ and $P_3 \subseteq \Omega[B]$.  
Choose such a partition with $B$ minimal.  
By (C2), $\Omega[B]$ contains two edges from $\delta(v)$ in distinct equivalence classes, and by our choice of $B$, neither of these edges is incident with $u$.  
Also by our choice of $B$, the subgraph $\Omega[B] - \{u,v\}$ is connected.  
It follows that $\Omega$ contains a $\Cc_4$ configuration.  
\end{proof}
This completes the proof in subcase 1(i).

\paragraph{Subcase (ii) \ $\Omega - v$ has an unbalanced cycle of length $>1$ but none containing $e_2$} \mbox{} 

Since two vertex disjoint paths linking the endpoints of $e_2$ and an unbalanced cycle would, by the theta property, yield an unbalanced cycle containing $e_2$, in this case $\Om-v$ is not 2-connected.  
We investigate the block structure of $\Om-v$ to show that $\Om$ contains a bad configuration.  

Suppose $\Psi$ is a leaf block of $\Omega - v$, containing cut-vertex $x$.  
By (C1) there is at least one edge between $v$ and $\Psi - x$.  
By (C2), either $\Psi$ is unbalanced or there exists an unbalanced cycle $C$ containing $v$ with length $>1$ with $C - v \subseteq \Psi$.  
With the goal of finding a bad configuration in mind, edges of $\Psi$ may be deleted or contracted to yield, in the former case, an unbalanced loop at $x$ and a link $vx$, or in the latter case, two $vx$ links forming an unbalanced cycle.  

Let $\Phi$ be the block of $\Om-v$ containing $e_2$.  
Suppose first that $\Phi$ is not a leaf block of $\Omega - v$.  
Then 
$\Phi$ contains two distinct cut-vertices $x, x'$.  
Choose a path in this block linking $x$ and $x'$ and containing $e_2$.  
Applying the argument of the previous paragraph to two leaf blocks of $\Om-v$, we find 
that $\Omega$ contains one of the configurations $\Cc_4$, $\Cc_4'$, or $\Cc_4''$.  

So suppose now that the block $\Phi$ of $\Om-v$ is a leaf block.  
After deleting unbalanced loops $\Phi$ is balanced, else $\Phi$ (and so $\Om-v$) would contain an unbalanced cycle containing $e_2$.  
Let $x$ be the cut vertex of $\Omega - v$ contained in $\Phi$, and let $S$ be the set of edges in $\delta(v)$ incident with a vertex of $\Phi - x$.  
Consider the biased graph $\Phi'$ obtained from $\Phi$ by deleting its unbalanced loops and adding vertex $v$ together with the edges in $S$.  
Vertex $v$ is a balancing vertex of $\Phi'$; let $\{S_1, \ldots, S_t\}$ be the partition of $S$ into the b-classes of $\delta(v)$ in $\Phi'$.  
Let $S_0$ be the set of loops in $\Phi$ not incident to $x$.  

Now construct an auxiliary graph similar to that appearing in subcase 1(i).  
Let $G$ be the graph obtained from $\Phi$ by adding vertices $x_0, x_1, \ldots, x_t$, and for $1 \leq i \leq t$ and every edge $zv \in S_i$ add an edge $z x_i$; for each unbalanced loop incident to a vertex $z$ add an edge $z x_0$.  
Finally, add a vertex $y$ that is adjacent to each vertex which is a neighbour of either $u$ or $w$.  
We claim that in $G$ there exist three vertex disjoint paths linking $\{x, x_0, \ldots, x_t\}$ to $\{u,w,y\}$.  
For suppose otherwise.  
Then by Menger's Theorem there exists a pair of subgraphs $G_1, G_2 \subseteq G$ whose edges partition $E(G)$ with $\{x, x_0, \ldots, x_t\} \subseteq V(G_1)$ and $\{u, w, y\} \subseteq V(G_2)$ and $|V(G_1) \cap V(G_2)| = 2$.  Let $Z = V(G_1) \cap V(G_2)$.  
If both vertices in $Z$ are in $\{x_0, x_1, \ldots, x_t\}$, then $\Omega - v$ would have no path linking $x$ and $u$, contradicting the fact that $\Phi$ is a block of $\Om-v$.  
Now either $x_0 \not\in Z$ or $x_0 \in Z$.  
If $x_0 \notin Z$, then in $\Om$ the biased subgraph induced by $E(G_2 - y)$ is a balanced subgraph meeting the rest of $\Om$ in just two vertices, contradicting (C2).  
But if $x_0 \in Z$, then the biased subgraph induced by $E(G_2 - y)$ meets the rest of $\Om$ in just one vertex, contradicting (C1).  
Hence the paths exist as claimed.  

We may assume that one of these three paths begins at vertex $x$ (otherwise choose a path from $x$ to $\{u,w,y\}$ modify a path appropriately).  
In $\Omega$ this gives us three internally disjoint paths $P_1, P_2, P_3 \subseteq \Phi'$ such that:  
\begin{enumerate}
\item $P_1,P_2$ start at $v$ or at a vertex of $\Phi$ incident with an unbalanced loop and end at $\{u,w\}$.  
\item at least one of $P_1,P_2$ starts at $v$, and if both start at $v$ their first edges are in distinct blancing classes.  
\item $P_3$ starts at $x$ and ends at $\{u,w\}$.  
\item at least one of $P_1, P_2, P_3$ ends at $u$ and one at $w$.
\end{enumerate}
Choose an unbalanced cycle $C$ of length $>1$ in $\Omega - v$ and choose two vertex disjoint paths $R, R'$ linking $C$ and $\{v,x\}$.  
Note that $C$ is not contained in $\Phi$ (as $\Phi$ without its unbalanced loops is balanced), and so $R, R'$ meet $\Phi$ only at $x$.  
First suppose that both $P_1$ and $P_2$ end at $u$ or both end at $w$.  
Consider the subgraph $H$ consisting of $C \cup R \cup R'$ together with $P_1 \cup P_2 \cup P_3$ and the edges $e_1, e_2$.  
If both $P_1, P_2$ begin at $v$ then $H$ contains a $\Cc_4$ configuration.  
Otherwise, one of these paths begins at a vertex incident with an unbalanced loop $f$.  
Adding $f$ to subgraph $H$, we find that $H$ contains $\Cc_4'$ configuration.  
So now suppose that $P_1$ ends at $u$ while both $P_2$ and $P_3$ end at $w$.  
Since $\Phi$ is a block of $\Om-v$, $\Phi - w$ contains a path $Q$ from $P_1 - v$ to $P_3 - \{v,w\}$.  
If $Q$ contains a vertex in $P_2$, then again the subgraph $H$ consisting of $C \cup R \cup R'$ together with $P_1 \cup P_2 \cup P_3 \cup \{e_1, e_2\}$ and possibly an unbalanced loop incident to an end of $P_1$ or $P_2$, contains either a $\Cc_4$ or $\Cc_4'$ configuration.  
Otherwise, $H$ contains either configuration $\Cc_5$ (if one of $P_1$ or $P_2$ does not begin at $v$ but is incident to an unbalanced loop) or $\Cc_8$ (if both $P_1$ and $P_2$ begin at $v$).

\paragraph{Subcase (iii) \ $\Omega - v$ has an unbalanced cycle containing $e_2$} \mbox{}

Let $C$ be an unbalanced cycle containing $e_2$.  
Choose two paths $P_1$, $P_2$ linking $v$ and $C$, disjoint except at $v$, say meeting $C$ at vertices $x_1, x_2$, respectively.  
Let $R$ be the $x_1$-$x_2$ path in $C$ containing $e_2$; let $R'$ be the $x_1$-$x_2$ path in $C$ avoiding $e_2$.  
If the cycle $P_1 \cup P_2 \cup R$ is unbalanced, then $\Om$ contains configuration $C_1$.  
So let us now assume that this does not occur for any unbalanced cycle containing $e_2$ | \ie, for every unbalanced cycle $C$ of $\Omega - v$ containing $e_2$ and every such pair $P_1, P_2$ of $v$-$C$ paths meeting only at $v$, the cycle formed by $P_1 \cup P_2$ and the path $R$ in $C$ traversing $e_2$ is balanced.  
Choose such subgraphs $C$, $P_1$ and $P_2$, with $P_1$ meeting $C$ at $x_1$ and $P_2$ meeting $C$ at $x_2$, so that the length of the path $R'$ in $C$ avoiding $e_2$ is minimum.  

Suppose $R'$ does not consist of a single edge.  
First suppose also that there exists a separation $(\Omega_1, \Omega_2)$ of $\Omega$ with $V(\Omega_1) \cap V(\Omega_2) = \{x_1,x_2\}$ with $R' \subseteq \Omega_1$ and $P_1 \cup P_2 \cup R \subseteq \Omega_2$.  
By choosing such a separation with $\Omega_1$ minimal, we may further assume that $\Omega_1 - \{x_1,x_2\}$ is connected and that there are no $x_1 x_2$ edges in $\Om_1$.  
By (C2), $\Omega_1$ is not balanced.  
If there is an unbalanced cycle in $\Omega_1 - x_1$, then $\Om$ contains a $\Cc_2$ configuration.  
Otherwise $x_1$ is a balancing vertex in $\Omega_1$.  
Since $\Om_1$ contains no $x_1$-$x_2$ edge 
and $\Omega_1 - \{x_1, x_2\}$ is connected, there is then an unbalanced cycle in $\Omega_1 - x_2$; again we find a $\Cc_2$ configuration.  
So now assume that no such separation exists: there is a path $Q$ from the interior of $R'$ to $(P_1 \cup P_2 \cup R) \setminus \{x_1, x_2\}$.  
If $Q$ first meets $P_1 \cup P_2 \setminus \{x_1, x_2\}$, then we find our choice of $P_1$ and $P_2$ did not minimise $R'$, a contradiction.  
Hence $Q$ avoids $(P_1 \cup P_2) \setminus \{x_1, x_2\}$ and meets $R$.  
Subgraph $Q \cup C$ is a theta. 
If the cycle in $Q \cup C$ containing $e_2$ different from $C$ is unbalanced, then again we did not choose $C$, $P_1$, and $P_2$ so as to minimise the length of $R'$, a contradiction.  
Therefore that cycle is balanced, and so the cycle $C'$ in $C \cup Q$ not containing $e_2$ is unbalanced.  
Choose an edge $e \in Q$.  
Contracting all edges of $C'$ but $e$, all but one edge of $R' \setminus C'$, all but edge $e_2$ of $R \setminus C'$, and all but one edge of each of $P_1$ and $P_2$, we find configuration $\Cc_2$.

So the path $R'$ must consist of a single $x_1 x_2$ edge $f$.  
Suppose first that $\{x_1,x_2\}$ does not separate $v$ from $C \setminus \{x_1,x_2\}$ and choose a path $Q$ from $(P_1 \cup P_2) \setminus \{x_1,x_2\}$ to $C \setminus \{x_1,x_2\}$.  
We claim that by the theta property, there exists a cycle in $P_1 \cup P_2 \cup Q \cup C$ containing both $e_2$ and $Q$ which is unbalanced, and in any case this yields a $\Cc_1$ configuration.  
To see this, recall that the cycle $P_1 \cup P_2 \cup C \setminus f$ is balanced.  
There are, up to symmetry and assuming $Q$ leaves from $P_1$, two cases to consider: (a) $Q$ is a $P_1$-$C$ path such that the cycle $D$ in $P_1 \cup Q \cup C$ containing $e_2$ and $Q$ contains $f$, or (b) does not contain $f$.  
In case (a), if $D$ is balanced, then the cycle in $P_1 \cup P_2 \cup (C \setminus f) \cup Q$ containing $Q$ and $e_2$ is unbalanced, and we find $\Cc_1$ contained in this cycle together with $C$ and $e_1$.  
If $D$ is unbalanced, then we find $\Cc_1$ in $P_1 \cup Q \cup C \cup \{e_1\}$.  
In case (b), if $D$ is balanced we find $\Cc_1$ by deleting the subpath of $P_1$ between $P_1 \cap Q$ and $x_1$.  
If $D$ is unbalanced, we find $\Cc_1$ in $P_1 \cup Q \cup C \cup \{e_1\}$.

Hence $\{x_1, x_2\}$ separates $v$ from $C$.  
Choose a separation $(\Omega_1,\Omega_2)$ of $\Omega$ with $V(\Omega_1) \cap V(\Omega_2) = \{x_1,x_2\}$ for which $C \subseteq \Omega_2$ and $v \in \Omega_1$, with $\Om_1$ minimal.  
Then $\Omega_1 - \{x_1,x_2\}$ is connected and $\Omega_1$ has no $x_1 x_2$ edge.  
If $\Om_1$ contains an unbalanced cycle $C'$ of length $>1$, then choosing a pair of vertex disjoint paths $Q, Q'$ linking $C'$ and $\{x_1, x_2\}$ and an application of the theta property yield an unbalanced cycle containing $e_2$ that is not $C$.  
But then $C' \cup Q \cup Q' \cup C \cup \{e_1\}$ contains a $\Cc_1$ configuration.  
Hence $\Om_1$ contains no unbalanced cycle of length $>1$; suppose $\Om_1$ contains an unbalanced loop $e \not= e_1$, say incident to $v'$.  
Since $N$ is 3-connected, $v' \not= v$.  
Since $\Omega$ is 2-connected, there is a path $Q$ from $v'$ to $(P_1 \cup P_2) \setminus v$.  
But now in $C \cup P_1 \cup P_2 \cup Q \cup \{e\}$ we find configuration $\Cc_2$. 

So $\Omega_1 - e_1$ is balanced.  
Now suppose that $V(\Omega_2) = \{x_1, x_2\}$.  
If there is a loop in $\Omega_2$ we have a $\Cc_2$ configuration.  
If there are at least three edges in $\Omega_2$ we have a $\Cc_3$ configuration (no two such edges form a balanced cycle since $N$ is 3-connected).  
So in this case $\Omega_2$ consists only of the two edges $e_2$ and $f$ (which form unbalanced cycle $C$).  
Since $\Om_1 - e_1$ is balanced, an unbalanced cycle in $\Om - f$ containing $e_2$, together with the theta property, would yield a $\Cc_1$ configuration.   
Hence $\Om - \{e_1, f\}$ is balanced.  
But this implies that $e_1$ and $f$ are in series in $N$, a contradiction since $N$ is 3-connected.  
So $|V(\Omega_2)| \geq 3$.  

We now claim that $\Omega_2$ contains an unbalanced cycle that does not contain both $x_1$ and $x_2$.  
Let $\Psi_0$ be a component of $\Omega_2 - \{x_1,x_2\}$ and let $\Psi$ be the subgraph of $\Omega_2$ consisting of $\Psi_0$ together with all edges between $x_i$ and $V(\Psi_0)$, for $i \in \{1, 2\}$.  
By (C2), $\Psi$ is unbalanced.  
Moreover, we may assume $x_1$ is a balancing vertex in $\Omega_2$, since if not we have the desire cycle.  
Consider the b-classes of $\delta_{\Omega_2}(x_1)$.  
Since $\Psi$ is not balanced, there are two edges in $\Psi$ in distinct b-classes, and since $\Psi - x_2$ is connected, this yields an unbalanced cycle in $\Omega_2$ not containing $x_2$, as desired.  

Without loss of generality, choose an unbalanced cycle $D \subseteq \Omega_2$ that does not contain $x_1$.  
If $D$ and $C$ share at most one vertex, we see that $P_1 \cup P_2 \cup C \cup D$ contains a $\Cc_2$ configuration.  
So $|V(C) \cap V(D)| \geq 2$.  
Let $Q$ be the maximal subpath of $C$ which contains $e_2$ and has no interior vertex in the set $V(D) \cup \{x_1,x_2\}$.  
By assumption at least one end of $Q$ must be in $V(D)$.  
If both ends of $Q$ are in $V(D)$ then $\Omega_2$ contains an unbalanced cycle $D'$ containing $e_2$ but not $x_1$.  
There are two vertex disjoint paths linking $D'$ and $\{x_1, x_2\}$ and these, together with $P_1 \cup P_2 \cup \{e_1,f\}$, contain a $\Cc_6$ configuration.  
So finally assume (by possibly interchanging $x_1$ and $x_2$) that one end of $Q$ is $x_1$ and the other is in $V(D)$.  
The $D$-$x_2$ path in $C$ avoids $Q$; this path together with $D$, $Q$, $f$, $P_1$, $P_2$, and $e_1$ contains a $\Cc_7$ configuration.  

This completes the proof of Case 1.

\subsubsection{Case 2. $e_1$ and $e_2$ are adjacent}

As before, let $v$ be the endpoint of $e_1$.  
Let $u$ be the other endpoint of $e_2$.  
Let $T_0$ be the standard block-cutpoint graph of $\Omega - v$.  If $u$ is a cut vertex of $\Omega - v$ then set $T = T_0$.  
Otherwise, let $T$ be the tree obtained by adding vertex $u$ to $T_0$ together with an edge between $u$ and the unique block of $\Omega - v$ containing $u$.  
View tree $T$ as rooted at $u$.  
Every block $\Psi$ of $\Omega - v$ is a vertex of $T$ and there is a unique path in $T$ from $\Psi$ to $u$.  
The next vertex of $T$ on this path from $\Psi$ is a vertex of $\Omega$, the \emph{parent} of $\Psi$.  
Note that the parent of a block of $\Om - v$ is always either a cut vertex of $\Omega - v$ or is $u$.  

\begin{claim} 
If $x$ is the parent of a block $\Psi$ of $\Omega - v$, then one of the following holds:
\begin{enumerate}
\item $\Psi$ contains no unbalanced cycle of length $>1$.
\item $x$ is balancing in $\Psi$ and there are exactly two b-classes in $\delta_\Psi(x)$.  
\end{enumerate}
\end{claim}

\begin{proof}[Proof of Claim] 
Let $\Psi'$ be the graph obtained from $\Psi$ by deleting all loops.  
If $\Psi'$ is balanced, (1) holds.  
Otherwise, suppose $x$ is not a balancing vertex of $\Psi'$ and choose an unbalanced cycle $C$ of $\Psi' - x$ and two paths $P_1, P_2$ from $x$ to $C$ that are disjoint except at $x$.  
Let $y_1, y_2$ be the respective ends of $P_1, P_2$ on $C$, and let $Q, Q'$ be the two paths in $C$ meeting just at $y_1$ and $y_2$.  
By the theta property, one of $P_1 \cup P_2 \cup Q$ or $P_1 \cup P_2 \cup Q'$ is unbalanced.  
Hence $P_1 \cup P_2 \cup C$ together with an $x$-$u$ path, $e_2$, and $e_1$, contains a $\Dd_2$ configuration.

So $x$ is balancing in $\Psi'$.  
If $\delta_\Psi(x)$ contains three b-classes, $\Om$ contains a $\Dd_2'$ configuration.  
Hence there are exactly two b-classes in $\delta_\Psi(x)$.  
If $\Psi$ contains an unbalanced loop not at vertex $x$, then an unbalanced cycle in $\Psi'$, together with this loop, an $x$-$u$ path, $e_2$, and $e_1$, contains a $\Dd_2$ configuration.  
\end{proof}

Call a block of $\Om-v$ as described in statement (1) of our claim a \emph{type 1} block, and a block as in statement (2), a \emph{type 2} block.  

\begin{claim} 
Every type 2 block of $\Om-v$ is a leaf of $T$. 
\end{claim}

\begin{proof}[Proof of Claim] 
Suppose there exists a type 2 block $\Psi$ of $\Omega - v$ that is not a leaf of $T$.  
Let $\Phi$ be a leaf block of $\Omega - v$ with parent $y$ such that the unique path in $T$ from $\Phi$ to $u$ contains $\Psi$.  
If $\Phi$ contains an unbalanced cycle, then $\Om$ contains a $\Dd_2$ configuration.  
So $\Phi$ is balanced.  
Let $\Phi^+$ be the biased subgraph of $\Om$ given by $\Phi$ together with $v$ and all edges between $v$ and $\Phi - y$.  
By (C2), $\Phi^+$ is unbalanced, so there is an unbalanced cycle $C$ in $\Phi^+$ containing $v$.  
Together with a $C$-$y$ path in $\Phi$, an unbalanced cycle $C'$ in $\Psi$, a $y$-$(C'-x)$ path and an $x$-$u$ path in $\Om-v$, we have a biased graph containing a $\Dd_3$ configuration.  
\end{proof}

Along with the structure we have determined of $\Omega - v$ comes knowledge of the biases of all cycles of $\Om-v$.  
We wish to extend this knowledge to $\Omega$.  

Let $\Om_0$ be the balanced biased subgraph of $\Om$ consisting of each type 1 block of $\Om-v$ without its unbalanced loops.  
By our second claim, $\Om_0$ is a connected balanced biased subgraph of $\Om-v$.  
Let $\Psi_1, \ldots, \Psi_m$ be the type 2 blocks of $\Omega - v$.  
For each $\Psi_i$, let $x_i$ be its parent vertex in $T$, and define $\Omega_i$ to be the subgraph of $\Omega$ consisting of $\Psi_i$ together with $v$ and all edges between $v$ and $\Psi_i - x_i$.  
Let $U$ be the set of all loops in $\Omega - v$.  
The subgraphs $E(\Omega_0), E(\Omega_1), \ldots, E(\Omega_m)$ are edge disjoint and together contain all edges in $E(\Omega)$ except for loops and some edges incident to $v$ (see Figure \ref{fig:flips_both_reps2}, at left).  

For every $1 \le i \le m$, vertex $x_i$ is balancing in $\Psi_i$; let $\{ A_i, B_i \}$ be the partition of $\delta_{\Psi_i}(x_i)$ into its two b-classes.  
Suppose $\Omega_i$ contains an unbalanced cycle $C$ disjoint from $x_i$.  
Choose two internally disjoint paths $P_1, P_2$ linking $x_i$ and $C$ for which $E(P_1) \cap A_i \not= \emptyset$ and $E(P_2) \cap B_i \not= \emptyset$.  
Then $C \cup P_1 \cup P_2 \cup \{e_1, e_2\}$ together with an $x_i$-$u$ path in $\Om - v$ contains a $\Dd_3$ configuration.  
Hence every $\Omega_i$ has $x_i$ as a balancing vertex.  
By Lemma \ref{lem:rerouting_along_a_bal_cycle} the b-classes in each $\delta_{\Om_i}(x_i)$ are $\{A_i, B_i \}$.  

Consider two edges $f, f' \in A_i$ or $f, f' \in B_i$ for some $1 \leq i \leq m$.  
Let $C$ (resp.\ $C'$) be a cycle containing $e_2$ and $f$ (resp.\ $f'$).  
The path $C- e_2$ ($C'-e_2$) is the union of a $u$-$x_i$ path $P$ $(P')$ and an $x_i$-$v$ path $Q$ $(Q')$.  
Applying Lemma \ref{lem:rerouting_along_a_bal_cycle} separately to $P \cup P'$ and $Q \cup Q'$, we conclude that $C$ and $C'$ have the same bias.  
Now suppose that for some $1 \leq i \leq m$, there is an unbalanced cycle containing $e_2$ and an edge in $A_i$ and another unbalanced cycle containing $e_2$ and an edge in $B_i$.  
Choose a cycle $C \subseteq \Psi_i$ that contains one edge in each of $A_i$ and $B_i$, a path $P$ in $\Omega_i$ from $v$ to $C-x_i$, and a $u$-$x_i$ path $Q$ in $\Omega_0$.  
It now follows that $P \cup Q \cup C \cup \{e_1,e_2\}$ contains a $\Dd_1$ configuration.  
Hence two such unbalanced cycles do not exist, and by possibly interchanging the names assigned to the sets $A_i, B_i$, we may assume that for every $1 \leq i \leq m$, every cycle in $\Omega$ containing $e_2$ and an edge of $A_i$ is balanced.  
By the theta property then, for every $1 \leq i < j \leq m$, every cycle of $\Omega$ containing an edge in $A_i$ and an edge in $A_j$ is balanced.

We now define a signature for $\Om$ that realises $\Bb$.  
We use a simpler biased graph $\Om'$ to model the biases of cycles in $\Om$ to do so.  
Let $\Om'$ be the biased graph obtained from $\Om$ as follows.   
For every $1 \leq i \leq m$ replace $\Omega_i$ with two edges $a_i, b_i$ with endpoints $x_i$ and $v$, with $a_i \in A_i$ and $b_i \in B_i$, and let the bias of each cycle of $\Om'$ be inherited from a corresponding cycle in $\Om$ in the obvious way.  
Now $\Omega' - v$ has no unbalanced cycle of length $> 1$, so by Observation \ref{obs:relabelling_when_there_is_a_balancing_vertex}, $\Om'$ is a $k$-signed graph.  
Moreover, by Observation \ref{obs:relabelling_when_there_is_a_balancing_vertex} there is a signature $\mathbf{\Sigma'} = \{U, \Sigma_1', \ldots, \Sigma_k'\}$ that realises the biases of cycles of $\Omega'$, where each set $\Sigma_j' \subseteq \delta_{\Om'}(v)$ and $U$ is the set of unbalanced loops of $\Om'$.  
Further, Observation \ref{obs:relabelling_when_there_is_a_balancing_vertex} allows us to assume that $e_2$ is not a member of any set in the signature $\mathbf{\Sigma'}$.  
Since for every $1 \leq i \leq m$, $e_2$ and $a_i$ are in the same b-class, none of the edges $a_i$ is in a member of the signature.  
Hence every $b_i$ is contained in some member $\Sigma_j'$ of $\mathbf{\Sigma'}$.  
Define a signature $\mathbf{\Sigma} = \{U, \Sigma_1, \ldots, \Sigma_k\}$ for $\Om$ as follows.  
For every $1 \le i \le m$, if $b_i \in \Sigma_j'$ put all edges in $B_i$ in $\Sigma_j$.  
If $e {=} vz \in \Sigma_j'$ is a edge incident to $v$ and a vertex $z \in \Omega_0$, put $e$ in $\Sigma_j$.  
The structural description we have of $\Om$ and the biases of its cycles implies $\Bb_{\mathbf{\Sigma}}=\Bb$.  
By Theorem \ref{prop:flip_operation}, the biased graph $\Gamma$ obtained by performing a twisted flip on $\Om$ has $F(\Gamma) \iso F(\Om)$ (Figure \ref{fig:flips_both_reps2}, at right).  
But in $\Gamma$ both $e_1$ and $e_2$ are represented as unbalanced loops, so $(N,L)$ is frame, a contradiction.  
\end{proof}

\begin{figure}[htbp] 
\begin{center} 
\includegraphics[scale=0.8]{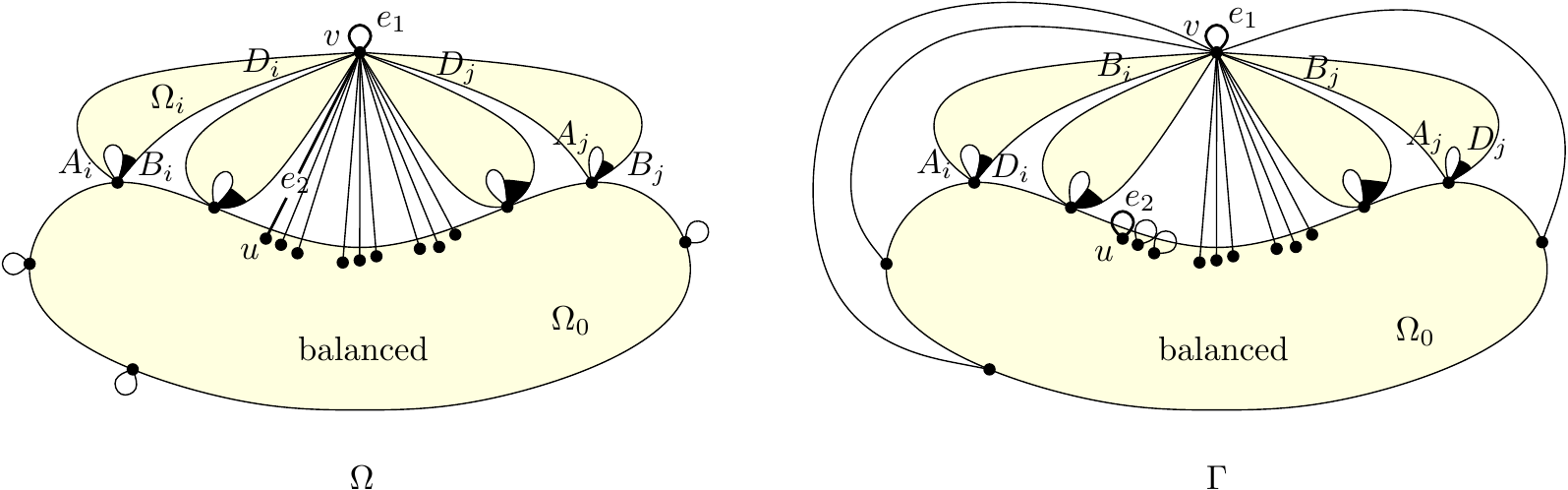}
\end{center} 
\caption{A twisted flip: $F(\Om) \iso F(\Gamma)$.}
\label{fig:flips_both_reps2} 
\end{figure}

\bibliography{Matroids1.bib} 

\begin{thebibliography}{10}

\bibitem{MR3417216}
Rong Chen, Matthew DeVos, Daryl Funk, and Irene Pivotto.
\newblock Graphical representations of graphic frame matroids.
\newblock {\em Graphs Combin.}, 31(6):2075--2086, 2015.

\bibitem{Matt_Luis_exminorsforbicircular}
M.~DeVos, L.~Goddyn, D.~Mayhew, and G.~Royle.
\newblock Excluded minors for bicircular matroids.
\newblock In preparation.

\bibitem{MR3267062}
Matt DeVos, Daryl Funk, and Irene Pivotto.
\newblock When does a biased graph come from a group labelling?
\newblock {\em Adv. in Appl. Math.}, 61:1--18, 2014.

\bibitem{MR0307951}
T.~A. Dowling.
\newblock A class of geometric lattices based on finite groups.
\newblock {\em J. Combinatorial Theory Ser. B}, 14:61--86, 1973.

\bibitem{MR3100270}
Torina Lewis, Jenny McNulty, Nancy~Ann Neudauer, Talmage~James Reid, and Laura
  Sheppardson.
\newblock Bicircular matroid designs.
\newblock {\em Ars Combin.}, 110:513--523, 2013.

\bibitem{MR0505702}
Laurence~R. Matthews.
\newblock Bicircular matroids.
\newblock {\em Quart. J. Math. Oxford Ser. (2)}, 28(110):213--227, 1977.

\bibitem{MR1892972}
Nancy~Ann Neudauer.
\newblock Graph representations of a bicircular matroid.
\newblock {\em Discrete Appl. Math.}, 118(3):249--262, 2002.

\bibitem{oxley:mt}
James~G. Oxley.
\newblock {\em Matroid theory}.
\newblock Oxford Science Publications. The Clarendon Press Oxford University
  Press, New York, 1992.

\bibitem{MR0317973}
J.~M.~S. Simoes-Pereira.
\newblock On subgraphs as matroid cells.
\newblock {\em Math. Z.}, 127:315--322, 1972.

\bibitem{MR815399}
Donald~K. Wagner.
\newblock Connectivity in bicircular matroids.
\newblock {\em J. Combin. Theory Ser. B}, 39(3):308--324, 1985.

\bibitem{MR1058551}
Thomas Zaslavsky.
\newblock Biased graphs whose matroids are special binary matroids.
\newblock {\em Graphs Combin.}, 6(1):77--93, 1990.

\bibitem{MR1088626}
Thomas Zaslavsky.
\newblock Biased graphs. {II}. {T}he three matroids.
\newblock {\em J. Combin. Theory Ser. B}, 51(1):46--72, 1991.

\bibitem{MR1273951}
Thomas Zaslavsky.
\newblock Frame matroids and biased graphs.
\newblock {\em European J. Combin.}, 15(3):303--307, 1994.

\bibitem{MR2017726}
Thomas Zaslavsky.
\newblock Biased graphs. {IV}. {G}eometrical realizations.
\newblock {\em J. Combin. Theory Ser. B}, 89(2):231--297, 2003.

\end{thebibliography}
\bibliographystyle{plain} 

\end{document}